\documentclass[11pt]{article}
\usepackage{amsmath,amssymb,theorem}
\usepackage{silence}
   \usepackage{hyperref}

\newtheorem{thm}{Theorem}[section]
\newtheorem{lemma}[thm]{Lemma}
\newtheorem{prop}[thm]{Proposition}
\newtheorem{cor}[thm]{Corollary}
{\theorembodyfont{\rmfamily}

\newtheorem{examp}[thm]{Example}
\newtheorem{examps}[thm]{Examples}
\newtheorem{rmk}[thm]{Remark}
}

\newcommand{\qed}{\hfill \mbox{\raggedright \rule{.07in}{.1in}}}

\newenvironment{proof}{\vspace{1ex}\noindent{\bf
Proof}\hspace{0.5em}}{\hfill\qed\vspace{1ex}}
\newenvironment{pfof}[1]{\vspace{1ex}\noindent{\bf Proof of
#1}\hspace{0.5em}}{\hfill\qed\vspace{1ex}}

\newcommand{\R}{\mathbb{R}}
\newcommand{\Z}{\mathbb{Z}}
\newcommand{\N}{\mathbb{N}}
 \newcommand{\E}{\mathbb{E}}
 \newcommand{\eps}{\epsilon}

 \renewcommand{\varkappa}{g}

\newcommand{\x}{x_{\epsilon}}
\newcommand{\y}{y_{\epsilon}}
\newcommand{\z}{z_{\epsilon}}
\newcommand{\hatx}{\hat x_{\epsilon}}

\newcommand{\tildez}{\tilde z_{\epsilon}}
\newcommand{\Lip}{{\operatorname{Lip}}}

\newcommand{\dist}{\operatorname{dist}}
\newcommand{\diam}{\operatorname{diam}}
\newcommand{\supp}{\operatorname{supp}}

\newcommand{\tildeJ}{{\tilde J}}

\numberwithin{equation}{section}

\newcommand{\cG}{{\mathcal{G}}}
\newcommand{\cM}{{\mathcal{M}}}
\newcommand{\cW}{{\mathcal{W}}}
\newcommand{\cS}{{\mathcal{S}}}

\newcommand{\SMALL}{\textstyle}
\newcommand{\BIG}{\displaystyle}

\begin{document}

\title{Martingale-coboundary decomposition for  \\ families of dynamical systems}

\author{A. Korepanov \hspace{2em}  Z. Kosloff  \hspace{2em} I. Melbourne \\[.75ex]
 {\small Mathematics Institute, University of Warwick, Coventry, CV4 7AL, UK}}

\date{4 August 2016; updated 23 August 2017.}

 \maketitle

\begin{abstract}
We prove statistical limit laws for sequences of Birkhoff sums of the type 
$\sum_{j=0}^{n-1}v_n\circ T_n^j$ where $T_n$ is a family of 
nonuniformly hyperbolic transformations.

The key ingredient is a new martingale-coboundary decomposition for
nonuniformly hyperbolic transformations which is useful already in the case when the family $T_n$ is replaced by a fixed transformation $T$, and which is particularly effective in
the case when $T_n$ varies with $n$.

 In addition to uniformly expanding/hyperbolic dynamical systems, our results include cases where 
 the family $T_n$ consists of intermittent maps, unimodal maps 
(along the Collet-Eckmann parameters), Viana maps, and externally forced dispersing billiards.

As an application, we prove a homogenization result for discrete fast-slow systems where the fast dynamics is generated by a family of nonuniformly hyperbolic transformations.
\end{abstract}

\section{Introduction}

The emergence of statistical and stochastic phenomena in deterministic dynamical systems is currently a very active area.  Topics of sustained interest include central limit theorems, invariance principles (weak and almost sure convergence to Brownian motion), moment estimates, and homogenization (whereby deterministic systems with multiple timescales converge to a stochastic differential equation).

One of the standard techniques for investigating such phenomena is the martingale-coboundary decomposition method of Gordin~\cite{Gordin69} which has seen extensive development in both the probability theory literature (for example~\cite{HallHeyde80, KV, MaxwellWoodroofe00, PhilippStout75}) and in the dynamical systems literature (for example~\cite{Liverani96,Tyran-Kaminska05,Viana97}).

In this paper, we introduce a new version of the Gordin method and show that it has significant advantages over previous versions when applied to a wide range of questions in nonuniformly hyperbolic dynamics.
Even in the case of a single nonuniformly hyperbolic transformation, there are
are advantages to the new approach
which seems more elementary and more powerful than the existing ones in the literature.
In addition, our method is well suited for studying sequences of Birkhoff sums of the form $S_n=\sum_{j=0}^{n-1}v_n\circ T_n^j$ where $T_n^{j+1}=T_n^j\circ T_n$
which arise naturally in averaging and homogenization problems. Here,
$T_n:\Lambda_n\to\Lambda_n$, $n\ge0$ is a sequence of measure-preserving transformations defined on probability spaces $(\Lambda_n,\mu_n)$.
It is assumed that the transformations $T_n$ are nonuniformly expanding/hyperbolic with uniform constants, but no restrictions are imposed on how $T_n$ varies with $n$.


In the case of a single nonuniformly hyperbolic map $T$, our method
applies directly to $T$ bypassing any induced limit theorems for the associated induced uniformly hyperbolic map.
Unlike other approaches~\cite{KV,Liverani96,MaxwellWoodroofe00,Tyran-Kaminska05}, no 
approximation arguments are required for the central limit theorem (CLT) and weak invariance principle (WIP) when the inducing time is not $L^3$.  For moment estimates, the method does not require special arguments when the inducing time is not $L^2$ (cf.~\cite{DedeckerMerlevede15,GouezelM14}).
In addition, we obtain a simple proof of an unexpected CLT
for systems with nonsummable decay of correlations due to~\cite{Gouezel04a},
whereas the previous proof relied on operator renewal theory and the Wiener lemma in noncommutative Banach algebras.

Still in the case of a single map $T$, our method is very well-adapted for obtaining a secondary martingale-coboundary decomposition for the square of the martingale in the decomposition mentioned above.
This enables control on sums of squares as is often required in more sophisticated limit laws.  To illustrate this, we consider an almost sure invariance principle with excellent error rates due to~\cite{CunyMerlevede15}, and show that our method of applying their results leads to stronger conclusions in certain examples.

The main advantage of the approach, however, is that it 
allows explicit control on various constants associated with each transformation $T$, making the method especially useful for sums of the form
$\sum_{j=0}^{n-1}v_n\circ T_n^j$.
This in turn has applications to fast-slow systems 
of the type considered in~\cite{KKMsub}.  Whereas~\cite{KKMsub} obtained rates of averaging, we prove results here on homogenization.

The remainder of this paper is organised as follows.
In Section~\ref{sec-NUE}, we establish the new martingale-coboundary decomposition for nonuniformly expanding maps and show how this implies moment estimates and the WIP.
In Section~\ref{sec-secondary}, we obtain a secondary martingale-coboundary decomposition and apply this to the almost sure invariance principle.
In Section~\ref{sec-NUElimit}, we derive limit laws for families of nonuniformly expanding maps.
This is extended to families of nonuniformly hyperbolic transformations in
Section~\ref{sec-NUHlimit}.
In Section~\ref{sec-homog}, we state and prove an abstract theorem on homogenization 
for discrete time fast-slow systems, generalising~\cite{GM13b}.
In Section~\ref{sec-homog-uniform}, we verify the hypotheses in Section~\ref{sec-homog} when the fast dynamics is given by a family of nonuniformly hyperbolic transformations.

\vspace{-2ex}
\paragraph{Notation}  
We write $\to_{\mu_n}$ to denote weak convergence with respect to a specific family of probability measures $\mu_n$ on the left-hand-side.  So $A_n\to_{\mu_n}A$ means that $A_n$ is a family of random variables on $(\Lambda_n,\mu_n)$
and $A_n\to_w A$.

For $J\in \R^{m\times n}$, we write
$|J|=\big(\sum_{i=1}^m\sum_{j=1}^n J_{ij}^2\big)^{1/2}$.

We use ``big O'' and $\ll$ notation interchangeably, writing $a_n=O(b_n)$ or $a_n\ll b_n$
if there is a constant $C>0$ such that
$a_n\le Cb_n$ for all $n\ge1$.
As usual, $a_n=o(b_n)$ means that $\lim_{n\to\infty}a_n/b_n=0$.

Recall that $v:\Lambda\to\R$ is a H\"older observable on a metric space $(\Lambda,d)$ if 
  $\|v\|_{\eta} = |v|_\infty + |v|_\eta<\infty$ where $|v|_\infty=\sup_\Lambda|v|$,
$|v|_\eta= \sup_{x\neq y} \frac{|v(x)-v(y)|}{d(x,y)^\eta}$.

\section{Martingale-coboundary decomposition for nonuniformly expanding maps}
\label{sec-NUE}

In this section, we prove our main theoretical result on martingale-coboundary decomposition for nonuniformly expanding maps.
The notion of nonuniformly expanding map is recalled in
Subsection~\ref{subsec-NUE}.   
The martingale-coboundary decomposition is carried out in Subsection~\ref{sec-primary}.  Subsection~\ref{sec-some} shows how certain limit laws 
follow from this decomposition.

\subsection{Nonuniformly expanding maps}
\label{subsec-NUE}

Let $(\Lambda,d_\Lambda)$ be a bounded metric space with
finite Borel measure $\rho$ and let $T:\Lambda\to \Lambda$ be a nonsingular
transformation. 
Let $Y\subset \Lambda$ be a subset of positive measure, and 
let $\alpha$ be an at most countable measurable partition
of $Y$ with $\rho(a)>0$ for all $a\in\alpha$.    We suppose that there is an
integrable 
{\em return time} function $\tau:Y\to\Z^+$, constant on each $a$ with
value $\tau(a)\ge1$, and constants $\lambda>1$, $\eta\in(0,1]$, $C_0,C_1\ge1$
such that for each $a\in\alpha$,
\begin{itemize}
\item[(1)] $F=T^\tau$ restricts to a (measure-theoretic) bijection from $a$  onto $Y$.
\item[(2)] $d_\Lambda(Fx,Fy)\ge \lambda d_\Lambda(x,y)$ for all $x,y\in a$.
\item[(3)] $d_\Lambda(T^\ell x,T^\ell y)\le C_0d_\Lambda(Fx,Fy)$ for all $x,y\in a$,
$0\le \ell <\tau(a)$.
\item[(4)] $\zeta_0=\frac{d\rho|_Y}{d\rho|_Y\circ F}$
satisfies $|\log \zeta_0(x)-\log \zeta_0(y)|\le C_1d_\Lambda(Fx,Fy)^\eta$ for all
\mbox{$x,y\in a$}.
\end{itemize}
Such a dynamical system $T:\Lambda\to \Lambda$ is called {\em nonuniformly expanding}.
We refer to $F=T^\tau:Y\to Y$ as the {\em induced map}.
(It is not required that $\tau$ is the first return time to $Y$.)
There is a unique absolutely continuous $F$-invariant probability measure $\mu_Y$
on $Y$ and $d\mu_Y/d\rho\in L^\infty$.  

Define the Young tower~\cite{Young99},
$\Delta=\{(y,\ell)\in Y\times\Z:0\le\ell\le \tau(y)-1\}$,
and the tower map $f:\Delta\to\Delta$  where
$f(y,\ell)=\begin{cases} (y,\ell+1), & \ell\le \tau(y)-2
\\ (Fy,0), & \ell=\tau(y)-1 \end{cases}$.
The projection
$\pi_\Delta:\Delta\to \Lambda$, $\pi_\Delta(y,\ell)=T^\ell y$, defines a semiconjugacy
from $f$ to $T$.  Define the ergodic $f$-invariant
probability measure
$\mu_\Delta=\mu_Y\times\{{\rm counting}\}/\int_Y \tau\,d\mu_Y$ on $\Delta$.
Then
$\mu=(\pi_\Delta)_*\mu_\Delta$ is an absolutely continuous ergodic  $T$-invariant probability measure.

\begin{rmk}  The above definition of nonuniformly expanding map covers many important classes of examples such as those mentioned in this paper.  
Indeed, it is generally true that nonuniform expansivity plus the existence of good statistical properties actually implies the existence of an inducing scheme satisfying the conditions above, see~\cite{AFLV11,AlvesLuzzattoPinheiro05}.  See~\cite{AlvesPinheiro10} for related results in the invertible setting (Section~\ref{sec-NUHlimit}).
\end{rmk}

In this section, we work with a fixed 
nonuniformly expanding map $T:\Lambda\to \Lambda$,
induced map $F=T^\tau:Y\to Y$, where $\tau\in L^p(Y)$ for some $p\ge1$,
and Young tower map $f:\Delta\to\Delta$.
The corresponding ergodic invariant probability measures are denoted $\mu$, $\mu_Y$ and $\mu_\Delta$.
Throughout, \(|\;|_p\) denotes the norm in
\(L^p(\mu)\) for functions on \(\Lambda\), in
\(L^p(\mu_Y)\) for functions on \(Y\), and in
\(L^p(\mu_{\Delta})\) for functions on \(\Delta\).
Also, $\|\;\|_\eta$ denotes the H\"older norm on $\Lambda$ and $Y$.

Although the map $T$ is fixed, the dependence of various constants on $T$ is important in later sections.
To simplify the statement of results in this section, we denote by \(C\) various constants 
depending continuously on $\diam\Lambda$, $C_0$, $C_1$, $\lambda$, $\eta$, $p$ and 
$|\tau|_p$.

\vspace{1ex}

Let $L:L^1(\Delta)\to L^1(\Delta)$ and $P:L^1(Y)\to L^1(Y)$ denote the transfer 
operators corresponding to $f:\Delta\to\Delta$ and $F:Y\to Y$.
(So $\int_\Delta Lv\,w\,d\mu_\Delta=\int_\Delta v\,w\circ f\,d\mu_\Delta$ for $v\in L^1(\Delta)$, $w\in L^\infty(\Delta)$, and 
$\int_Y Pv\,w\,d\mu_Y=\int_Y v\,w\circ F\,d\mu_Y$ for $v\in L^1(Y)$, $w\in L^\infty(Y)$.)

  Let \(\zeta = \frac{d \mu_Y}{d\mu_Y \circ F}\).
Given $y\in Y$, let
$y_a$ denote the unique $y_a\in a$ with $Fy_a=y$.
Then we have the pointwise expressions for $P$ and $L$,
\begin{align} \label{eq-PL}
(P\psi)(y)=\sum_{a\in\alpha} \zeta(y_a)\psi(y_a),
\quad
    (L \psi)(y, \ell) = 
    \begin{cases}
      \sum_{a \in \alpha} \zeta(y_a) \psi(y_a, \tau(y_a)-1), & \ell=0 \\
      \psi(y, \ell-1), & \ell \ge 1
    \end{cases}.
\end{align}

\begin{prop} \label{prop-zeta}
    $\zeta(x) \le C \mu_Y(a)$ and
    $| \zeta(x) - \zeta(y)|
    \le C \mu_Y(a) d_\Lambda(Fx,Fy)^\eta$
for all $x,y\in a$, $a\in\alpha$.
\end{prop}

\begin{proof}  
  By \cite[Propositions~2.3 and 2.5]{KKMnote},
  \(
    \big| \log \zeta(x) - \log \zeta(y) \big|
    \ll d_\Lambda(Fx,Fy)^\eta.
  \)
  In particular, \(\zeta(x)/ \zeta(y) \ll 1\).
  Hence
\[
\SMALL \mu_Y(a)=\int_Y 1_a \, d\mu_Y=\int_Y P 1_a \, d\mu_Y\ge\inf P1_a = \inf_a \zeta \gg \sup_a \zeta,
\]
and so $\zeta(y)\ll \mu_Y(a)$.

Next, we note the inequality $|s-t|\le\max\{s,t\}|\log s-\log t|$ which is valid for all $s,t>0$.
Hence
$|\zeta(x)-\zeta(y)|\ll \sup_a\zeta d_\Lambda(Fx,Fy)^\eta\ll \mu_Y(a) d_\Lambda(Fx,Fy)^\eta$.
\end{proof}

\subsection{The primary martingale-coboundary decomposition}
\label{sec-primary}

Let \(v :\Lambda \to \R^d\) be H\"older with $\int_\Lambda v\,d\mu=0$, and define \(\phi=v\circ\pi_\Delta:\Delta\to\R^d\).  
Define the 
\emph{induced observable} $\phi':Y\to\R^d$ by \(\phi'(y) = \sum_{\ell=0}^{\tau(y)-1} \phi(y,\ell)\).

\begin{prop}
  \label{prop-iLip}
  \(\|P \phi'\|_\eta \le C \|v\|_\eta \).
\end{prop}

\begin{proof}
  Let \(x,y \in Y\), \(a \in \alpha\), with corresponding preimages $x_a,y_a\in a$.  Then
  \begin{align}
    \label{eq-sbyz}
    | \phi'(x_a) - \phi'(y_a)| 
    \leq \sum_{\ell=0}^{\tau(a)-1} \big|v(T^\ell x_a) - v(T^\ell y_a)\big|
    \ll|v|_\eta \tau(a) d_\Lambda(x,y)^\eta.
  \end{align}
Also $|\phi'|\le |v|_\infty\,\tau$.
By~\eqref{eq-PL} and Proposition~\ref{prop-zeta},
  \begin{align*}
    |(P \phi')(x) - (P \phi')(y)|
 & \le \sum_{a \in \alpha} 
      |\zeta(x_a) - \zeta(y_a)|\,|\phi'(x_a)|
 + \sum_{a \in \alpha} 
       \zeta(y_a)|\phi'(x_a)-\phi'(y_a)|
    \\ & \ll \|v\|_\eta \sum_{a \in \alpha} \mu_Y(a) \tau(a) d_\Lambda(x,y)^\eta
    \ll \|v\|_\eta d_\Lambda(x,y)^\eta.
  \end{align*}
Hence $|P\phi'|_\eta \ll \|v\|_\eta$.
Similarly, $|P\phi'|_\infty \ll |v|_\infty$.
\end{proof}

Define \(\chi', m' : Y \to \R^d\) 
as follows: 
\[ 
\SMALL \chi' = \sum_{k=1}^\infty P^k \phi', \qquad
\phi' = m' + \chi' \circ F - \chi'.
\]
By Proposition~\ref{prop-iLip}
and~\cite[Corollary~2.4 and Proposition~2.5]{KKMnote},
\begin{align*} 
\SMALL  \|\chi'\|_\eta & \le \sum_{k=0}^\infty\|P^kP\phi'\|_\eta\ll \|P\phi'\|_\eta
\ll \|v\|_\eta, \\
\nonumber   |m'|_p & \le |\phi'|_p+2|\chi'|_\infty
\le |v|_\infty |\tau|_p+2|\chi'|_\infty\ll \|v\|_\eta.
\end{align*}
Define \(\chi,m : \Delta \to \R^d\) by
\[
  \chi(y, \ell) = 
    \chi'(y)+ \sum_{k=0}^{\ell-1} \phi(y, k),
  \qquad
  m(y, \ell) = \begin{cases}
    0, & \ell \leq \tau(y)-2 \\
    m'(y), & \ell=\tau(y)-1
  \end{cases}.
\]

\begin{prop} \label{prop-mr}
  \(|m|_p \leq C\|v\|_\eta\) and
  \(|\chi|_{p-1} \leq C \|v\|_\eta\).
\end{prop}

\begin{proof}
Compute that 
$\int_\Delta |m|^p\,d\mu_\Delta 
= |\tau|_1^{-1}\int_Y \sum_{\ell=0}^{\tau(y)-1}|m(y,\ell)|^p\,d\mu_Y
= |\tau|_1^{-1}\int_Y |m'(y)|^p\,d\mu_Y 
\le  |m'|_p^p
\ll \|v\|_\eta^p$.
Similarly, $|\chi(y,\ell)|\le |\chi'|_\infty+\ell|v|_\infty\ll \tau(y)\|v\|_\eta$ yielding the estimate for $\chi$.
\end{proof}

\begin{prop} \label{prop-m}
  \(\phi = m + \chi \circ f - \chi\) and
  \(m \in \ker L\).
\end{prop}

\begin{proof}
If $\ell\le\tau(y)-2$, then 
\[
\chi\circ f(y,\ell)-\chi(y,\ell)=\chi(y,\ell+1)-\chi(y,\ell)=\phi(y,\ell)=\phi(y,\ell)-m(y,\ell).
\]
For $p=(y,\tau(y)-1)$,
\begin{align*}
\chi\circ f(p)-  \chi(p)
 & =\chi(Fy,0)-\chi(y,\tau(y)-1) 
=\chi'(Fy)-\chi'(y)-{\SMALL\sum_{k=0}^{\tau(y)-2}}\phi(y,k)
\\ & =\phi'(y)-m'(y)-{\SMALL\sum_{k=0}^{\tau(y)-2}}\phi(y,k)
=\phi(p)-m(p).
\end{align*}
Hence $\phi=m+\chi\circ f-\chi$.

  By definition,  $Pm' =P\phi'-\chi'+P\chi'\equiv 0$.
  Using~\eqref{eq-PL}, observe that \((Lm)(y,\ell) =m(y,\ell-1)= 0\) if \(\ell \ge 1\), and 
  \[
    \SMALL (Lm)(y,0) = \sum_{a \in \alpha} \zeta (y_a)m(y_a, \tau(y_a)-1) 
    = \sum_{a \in \alpha} \zeta (y_a)m'(y_a) 
    = (Pm')(y) = 0.
  \]
Hence $m\in\ker L$.
\end{proof}

\begin{prop} \label{prop-chiae}
$\max_{\,0\le k\le n}|\chi\circ f^k|=o(n^{1/p})$ a.e.
\end{prop}

\begin{proof}
Since $\tau\in L^p$, it follows from the ergodic theorem that
$\tau\circ F^n=o(n^{1/p})$ a.e., and hence that
$\max_{\,0\le k\le n}\tau\circ F^k=o(n^{1/p})$ a.e.

Next, $|\chi(y,\ell)|\le |\chi'|_\infty+\ell|v|_\infty\ll \tau(y)\|v\|_\eta$.
For any $(y,\ell)\in\Delta$ and $n\ge0$, there exists
$k\in\{0,\dots, n\}$ and $\ell'\in\{0,\dots,\tau(F^ky)-1\}$ such that
$f^n(y,\ell)=(F^ky,\ell')$.
Hence $|\chi(f^n(y,\ell))|\ll \|v\|_\eta\max_{\,0\le k\le n}\tau(F^ky)$ and so
$\max_{\,0\le k\le n}|\chi(f^k(y,\ell))|\ll \|v\|_\eta\max_{\,0\le k\le n}\tau(F^ky)=o(n^{1/p})$ a.e.
\end{proof}

No uniformity of constants is claimed in Proposition~\ref{prop-chiae}.
It is straightforward to show that $\big|\max_{\,0\le k\le n}|\chi\circ f^k|\big|_{p-1}\le C\|v\|_\eta n^{1/(p-1)}$ where $C$ is a uniform constant.  However, for various purposes (such as optimal moment estimates in Corollary~\ref{cor-moment}) we require the following more delicate estimate.

\begin{prop} \label{prop-chi}
    $\big| \max_{1 \leq k \leq n} |\chi \circ f^k - \chi| \big|_p 
    \le C \|v\|_\eta n^{1/p}$.
  Moreover,
  \[ \SMALL
    \big| \max_{1 \leq k \leq n} |\chi \circ f^k - \chi| \big|_p 
    \le C \|v\|_\eta ( n^{1/q} + n^{1/p} |1_{\{\tau\ge n^{1/q}\}}\tau|_p )
    \quad \text{for all } n \ge 0,\,q\ge p.
  \]
\end{prop}

\begin{proof}
  Define \(t_a = |1_{\{\tau \geq a\}} \tau|_p \), $a\ge0$.
Then
  \begin{align}
    \label{eq-asy2} \nonumber
      \SMALL \sum_{k \geq n} k^{p-1} \mu_Y (\tau \geq k) 
      &= \SMALL \sum_{k \geq n} \sum_{j \geq k} k^{p-1} \mu_Y(\tau=j)
      \\ & \SMALL= \sum_{j \geq n} \mu_Y(\tau=j) \sum_{k=n}^j k^{p-1} 
       \leq \sum_{j \geq n} j^p \mu_Y(\tau=j)
      = t_n^p.
  \end{align}

  Let \(\Delta_n = \{(y, \ell) \in \Delta : \ell=n\}\)
  and \(A_n = \{(y, \ell) \in \Delta : \ell < \tau(y)-n\} \).
  Then \(\mu_\Delta(\Delta_n) = |\tau|_1^{-1} \mu_Y(\tau \geq n)\) and
  \(\mu_\Delta(A_n) = \mu_\Delta(\cup_{k \geq n} \Delta_k)
  = |\tau|_1^{-1} \sum_{k \geq n} \mu_Y(\tau \geq k)\).
  By~\eqref{eq-asy2},
  \[
    \SMALL
    n^{p-1} \mu_\Delta(A_n)
    = n^{p-1} |\tau|_1^{-1} \sum_{k \geq n} \mu_Y(\tau \geq k)
    \leq \sum_{k \geq n} k^{p-1} \mu_Y(\tau \geq k)
    \leq t_n^p.
  \]
  If \((y, \ell) \in A_n\), then \(\max_{1 \leq k \leq n}
  \big|(\chi \circ f^k - \chi)(y, \ell)\big| \leq n |v|_\infty\).
  Therefore
  \begin{align} \nonumber
    \label{eq-lbakf}
\SMALL 
    \big|1_{A_n}
      \max_{1 \leq k \leq n} |\chi \circ f^k - \chi|
    \big|_p 
    & \leq n |v|_\infty [\mu_\Delta(A_n)]^{1/p}
    \\ & = n^{1/p} |v|_\infty [n^{p-1} \mu_\Delta(A_n)]^{1/p}
     \leq |v|_\infty n^{1/p} t_n.
  \end{align}

For all $(y,\ell)\in\Delta$,
  we have $|\chi(y,\ell)| \ll \tau(y)\|v\|_\eta$
  and so $|\chi\circ f^k|\ll \|v\|_\eta\max_{\,0\le j\le k}\tau\circ F^j$.
  Let \(a > 0\) and denote \(\tau_a = 1_{\{\tau > a\}}\tau\).
  Since \(\tau^p \leq a^p + \tau_a^p\), 
  \begin{align}
    \label{eq-ff9ns} \nonumber
      \SMALL
      \|v\|_\eta^{-p} & \SMALL \max_{1\le k\le n}|\chi(f^k(y, \ell))-\chi(y, \ell)|^p
      \leq 2^p \|v\|_\eta^{-p} \max_{0\le k\le n}|\chi(f^k(y, \ell))|^p
      \\ & \SMALL \ll  \max_{\,0\le k\le n} \tau^p (F^ky)
      \leq  a^p +  {\SMALL \sum_{0 \leq k \leq n}} \tau_a^p (F^ky). 
    \end{align}

  Suppose that $\psi:\Delta\to\R$ has the form
  $\psi(y,\ell)=\psi_0(y)$ where $\psi_0:Y\to\R$.
  Then
  \begin{align} 
    \label{eq-j8ia}
    \SMALL
    \int_{\Delta\setminus A_n}|\psi|\,d\mu_\Delta
    =|\tau|_1^{-1}\int_Y\min\{\tau,n\}|\psi_0|\,d\mu_Y
    \le \int_Y\min\{\tau,n\}|\psi_0|\,d\mu_Y.
  \end{align}

  Taking $v=1$ in Proposition~\ref{prop-iLip} (resulting in $\phi=1$, $\phi'=\tau$)
  and using that $P$ is a contraction
  yields the estimate $|P^k\tau|_\infty\le |P\tau|_\infty\ll \|1\|_\eta=1$
  for all \(k \geq 1\).
  Then by equations~\eqref{eq-ff9ns} and~\eqref{eq-j8ia},
  \begin{align*}
    \SMALL
    \|v\|_\eta^{-p} & \SMALL \int_{\Delta\setminus A_n}\max_{1\le k\le n}|\chi\circ f^k-\chi|^p
      \,d\mu_\Delta
    \ll \SMALL a^p + \sum_{0 \leq k \leq n} 
      \int_{\Delta\setminus A_n} \tau_a^p (F^ky)
      \,d\mu_\Delta(y,\ell)
    \\ & \SMALL \leq a^p + \sum_{0 \leq k \leq n} 
      \int_Y \min\{\tau, n\} \, \tau_a^p \circ F^k \, d\mu_Y
    \leq a^p + n |\tau_a^p|_1 + \sum_{k=1}^n | \tau \, \tau_a^p \circ F^k |_1
    \\ & \SMALL= a^p + n |\tau_a^p|_1 + \sum_{k=1}^n | P^k \tau \, \tau_a^p |_1
    \ll \SMALL a^p + n |\tau_a^p|_1 = a^p + n t_a^p.
  \end{align*}
Hence
 \begin{align*}
   \big| 1_{\Delta \setminus A_n} \max_{1 \leq k \leq n} |\chi \circ f^k - \chi| \big|_p 
   & \ll \|v\|_\eta ( a^p + nt_a^p )^{1/p}
   \leq \|v\|_\eta ( a + n^{1/p}t_a ).
 \end{align*}
 We take \(a = n^{1/q}\). 
 Combining with~\eqref{eq-lbakf} and using
 \(t_n \leq t_{n^{1/q}}\) completes the proof.~
\end{proof}

\begin{cor} \label{cor-chi}
$\big| \max_{1 \leq k \leq n} |\chi \circ f^k - \chi| \big|_p =o(n^{1/p})$.
\qed
\end{cor}

The next result justifies calling
$\phi=m+\chi\circ f-\chi$ a {\em martingale-coboundary decomposition}.
Let $U$ denote the Koopman operator corresponding to $f$, i.e.\ $Uv=v\circ f$.  

\begin{prop}   \label{prop-mart}
Fix $n\ge1$.
Let $\cM$ denote the underlying $\sigma$-algebra on $(\Delta,\mu_\Delta)$ and 
define $\cG_j=f^{-(n-j)}\cM$, $1\le j\le n$.
Then $\{m\circ f^{n-j},\,\cG_j;\,1\le j\le n\}$ is a sequence of 
{\em martingale differences}.  That is, $\cG_1\subset\dots\subset \cG_n$,
$m\circ f^{n-j}$ is $\cG_j$-measurable for each $j$, 
and $\E(m\circ f^{n-j}|\cG_{j-1})=0$ for each $j$. 
\end{prop}

\begin{proof}
Since $f^{-1}\cM\subset\cM$, it follows that $\cG_j\subset\cG_{j+1}$.
Measurability of $m\circ f^{n-j}$ with respect to $\cG_j$ is clear.
It is standard, and easy to check, that $UL=\E(\,\cdot\, |f^{-1}\cM)$.
Hence
\begin{align*}
\E(m\circ f^{n-j}|\cG_{j-1})
& =\E(m|f^{-1}\cM)\circ f^{n-j}
 =(ULm)\circ f^{n-j}=0,
\end{align*}
since $m\in\ker L$.
\end{proof}

\subsection{Some limit theorems}
\label{sec-some}

Suppose that $v:\Lambda\to\R^d$ 
is H\"older and $\int_\Lambda v\,d\mu=0$.
By the results from Subsection~\ref{sec-primary} we have the decomposition
$v\circ\pi_\Delta=\phi= m+\chi\circ f-\chi$,
where $m,\,\chi$ satisfy the estimates in Propositions~\ref{prop-mr} and~\ref{prop-chi}.

\begin{cor}[Moments]  \label{cor-moment}
If $p\le 2$, then
$\big|\max_{j\le n}|\sum_{k=0}^{j-1}v\circ T^k|\big|_p\le C\|v\|_\eta n^{1/p}$ 
and
$\big|\max_{j\le n}|\sum_{k=0}^{j-1}m\circ f^k|\big|_p\le C\|v\|_\eta n^{1/p}$ 
for all \mbox{$n\ge1$}.

If $p\ge 2$, then
$\big|\max_{j\le n}|\sum_{k=0}^{j-1}v\circ T^k|\big|_{2(p-1)}\le C\|v\|_\eta n^{1/2}$ 
for all \mbox{$n\ge1$}.
\end{cor}

\begin{proof} Since $\{m\circ f^{n-j};\,1\le j\le n\}$ is a sequence of martingale differences with respect to the filtration $\cG_j=f^{n-j}\cM$ for each $n\ge1$ (Proposition~\ref{prop-mart}), it follows from Burkholder's inequality~\cite{Burkholder73} and Proposition~\ref{prop-mr} that for $p\le2$,
\[
\SMALL \big|\max_{j\le n}|\sum_{k=1}^j m\circ f^{n-k}|\big|_p\ll n^{1/p}|m|_p \ll n^{1/p}\|v\|_\eta.
\]
Writing $\sum_{k=0}^{j-1}m\circ f^k=\sum_{k=1}^n m\circ f^{n-k}-
\sum_{j=1}^{n-j}m\circ f^{n-k}$, we obtain that
$\big|\max_{j\le n}|\sum_{k=0}^{j-1} m\circ f^k|\big|_p\ll n^{1/p}\|v\|_\eta$.
Combining this with Proposition~\ref{prop-chi} yields
$\big|\max_{j\le n}|\sum_{k=0}^{j-1} \phi\circ f^k|\big|_p\ll n^{1/p}\|v\|_\eta$ and the result for $p\le2$ follows.

When $p\ge2$, we use Rio's inequality~\cite{Rio00} following~\cite{MN08}.
See~\cite[Proposition~7]{MerlevedePeligradUtev06} for a statement of Rio's inequality.
Let $X_j=\phi\circ f^{n-j}$.
For $1\le j\le \ell\le n$, by Proposition~\ref{prop-mart},
$\sum_{k=j}^\ell\E(X_k|\cG_j)=m\circ f^{n-j}+\E(\chi\circ f^{n+1-\ell}|\cG_j)-\chi\circ f^{n-j}$.
By Proposition~\ref{prop-mr}, $\max_{1\le j\le\ell\le n}|\sum_{k=j}^\ell\E(X_k|\cG_j)|_{p-1}\ll \|v\|_\eta$.
Hence $\max_{1\le j\le \ell\le n}|X_j\sum_{k=j}^\ell \E(X_k|\cG_j)|_{p-1}
\le |\phi|_\infty\max_{1\le j\le \ell\le n}|\sum_{k=j}^\ell \E(X_k|\cG_j)|_{p-1}
\ll \|v\|_\eta^2$.  The result follows by Rio's inequality.
\end{proof}

\begin{rmk} \label{rmk-moment}
The moment estimates for $p\ge2$ were first obtained in~\cite{MN08} and
the results for $p<2$ are due to~\cite{DedeckerMerlevede15,GouezelM14}.

Corollary~\ref{cor-moment} is easily seen to be optimal given the formulation of our results in this paper in terms of the integrability of the return time $p$.
Often a tail estimate of the form
$\mu_Y(\tau>n)=O(n^{-p})$ is available, and this gives rise to some interesting subtleties; such issues are also resolved in~\cite{DedeckerMerlevede15,GouezelM14}.
On the other hand, these references do not explicitly address the uniformity of the constant $C$ which is required in later sections.
\end{rmk}

\begin{cor}[Covariance] \label{cor-Sigma}  Suppose that $p\ge2$.
Then
$\lim_{n\to\infty}n^{-1}\int_\Lambda
(\sum_{j=0}^{n-1}v\circ T^j)
(\sum_{j=0}^{n-1}v\circ T^j)^T d\mu=\int_\Delta m\,m^T d\mu_\Delta$.
\end{cor}

\begin{proof}
Write $S_nv=\sum_{j=0}^{n-1}v\circ T^j$ and similarly define $S_n\phi$ 
and $S_nm$.
Since $m\in\ker L$, 
$\int_\Delta S_nm\,S_nm^T\,d\mu_\Delta =n\int_\Delta m\,m^T\,d\mu_\Delta$.

By Corollary~\ref{cor-moment}, $|S_n\phi|_2\ll n^{1/2}\|v\|_\eta$ and
$|S_nm|_2\ll n^{1/2}\|v\|_\eta$.  Hence,
\begin{align*}
   \Big|
&     n^{-1}\int_\Lambda S_nv\,S_nv^T \, d\mu -    \int_\Delta m\,m^T\,d\mu_\Delta 
  \Big| 
    =n^{-1}\Big|
    \int_\Delta S_n\phi\, S_n\phi^T \, d\mu_\Delta - \int_\Delta S_nm\,S_nm^T\,d\mu_\Delta 
  \Big| \\
& \qquad\qquad  \le n^{-1}|S_n\phi\, S_n\phi^T-S_nm\,S_nm^T|_1
\le 
n^{-1}(|S_n\phi|_2+|S_nm|_2)|S_n\phi-S_nm|_2
\\ &  \qquad\qquad \ll n^{-1/2}|\chi\circ f^n-\chi|_2\|v\|_\eta\to0
\end{align*}
by Corollary~\ref{cor-chi}.
\end{proof}

For $n\ge1$, define the process
  $W_n(t)  = n^{-1/2}\sum_{j=0}^{[nt]-1} v     \circ T^j$ on $(\Lambda,\mu)$.

\begin{cor}[WIP] \label{cor-WIP}
Suppose that $p\ge2$.
Then $W_n\to_\mu W$ where $W$ is Brownian motion
with covariance $\Sigma=\lim_{n\to\infty}n^{-1}\int_\Lambda 
(\sum_{j=0}^{n-1}v\circ T^j)
(\sum_{j=0}^{n-1}v\circ T^j)^T\,d\mu$.
\end{cor}

\begin{proof}
By Corollary~\ref{cor-Sigma}, $\Sigma=\int_\Delta m\,m^T\,d\mu_\Delta
=\int_\Delta UL(m\,m^T)\,d\mu_\Delta$.
By the ergodic theorem, $n^{-1}\sum_{j=0}^{[nt]-1}\{UL(mm^T)\}\circ f^j\to t\Sigma$  a.e.\ as $n\to\infty$ for all $t>0$.
Hence we can apply Theorem~\ref{thm-mda} to the process
$M_n(t)=n^{-1/2}\sum_{j=0}^{[nt]-1} m\circ f^j$ to deduce that
$M_n\to_{\mu_\Delta}W$.

Next define
$\Phi_n(t)=n^{-1/2}\sum_{j=0}^{[nt]-1} \phi\circ f^j$.
For any $T>0$,
\begin{align*}
\SMALL \big|\sup_{t\in[0,T]}|\Phi_n(t)-M_n(t)|\big|_2
\le n^{-1/2}\big|\max_{1\le j\le nT}|\chi\circ f^j-\chi|\big|_2\to0,
\end{align*}
by Corollary~\ref{cor-chi}.
Hence $\Phi_n\to_{\mu_\Delta}W$.
Finally, $\pi_\Delta$ is a measure-preserving semiconjugacy, so the result follows.
\end{proof}

\begin{rmk}  As mentioned in the introduction, previous methods~\cite{Liverani96,MaxwellWoodroofe00,Tyran-Kaminska05} require special techniques for $p\le3$.
In particular, a sequence of martingale approximations is needed, whereas we work with a single martingale-coboundary decomposition.

Alternatively,~\cite{Gouezel07,MN05} obtained a single martingale-coboundary decomposition at the level of the induced map.  The resulting CLT/WIP can then be lifted back to the original system by~\cite{MT04,MZ15}.
\end{rmk}

Finally, we show how to recover a result of~\cite{Gouezel04a} where the WIP holds somewhat unexpectedly.  Our method of proof is significantly simpler than
in~\cite{Gouezel04a}.  On the other hand,~\cite{Gouezel04a} also obtained unexpectedly fast decay of correlations in this situation.

\begin{cor}[WIP with $p=1$]
Suppose that $\tau:Y\to\Z^+$ is the first return to $Y$ and that
$\supp v\subset Y$.  (We continue to suppose that $v$ is H\"older with mean zero.)
Then the WIP above holds for all $p\ge1$.
\end{cor}

\begin{proof}
The assumptions on $\tau$ and $v$ ensure that $\phi(y,\ell)=0$ for $\ell\ge1$.
But then the definition of $\chi$ reduces to
  $\chi(y, \ell) = \begin{cases}
    \chi'(y), & \ell=0 \\
    \chi'(y)+ \phi(y, 0), & \ell \ge 1
  \end{cases}$,
and it follows that $|\chi|_\infty \ll \|v\|_\eta$ and hence
that $|m|_\infty \ll \|v\|_\eta$.
The arguments above go through (with numerous simplifications).
\end{proof}

\begin{rmk} \label{rmk-phi}
The results in Subsections~\ref{sec-primary} and~\ref{sec-some} were proved for observables $\phi=v\circ\pi_\Delta$ where $v:\Lambda\to\R^d$ is H\"older and mean zero.
It is easy to check that the only properties of $\phi$ that were used are
(i) $\int_\Delta\phi\,d\mu_\Delta=0$,
(ii) $\phi\in L^\infty$,
(iii) $\|P\phi'\|_\eta<\infty$.
For such observables $\phi:\Delta\to\R^d$, all the results go through with
$\|v\|_\eta$ replaced by $|\phi|_\infty+\|P\phi'\|_\eta$.
\end{rmk}

\section{Secondary martingale-coboundary decomposition and the ASIP}
\label{sec-secondary}

In this section, we derive a secondary  martingale-coboundary decomposition
for nonuniformly expanding maps.  As an illustration of its utility, we obtain
an ASIP for nonuniformly expanding maps with improved error rates over those in the literature.

We continue to suppose that $v:\Lambda\to\R^d$ is H\"older with $\int_\Lambda v\,d\mu=0$
and that $\phi=v\circ\pi_\Delta:\Delta\to\R^d$.
In addition, we suppose that $p\ge2$.

Let 
\[
\phi = m + \chi \circ f - \chi,\quad  \phi' = m' + \chi' \circ f - \chi'
\]
be the decompositions from Subsection~\ref{sec-primary}. 
Define \(\breve{\phi} : \Delta \to \R^{d \times d}\),
\[
\SMALL \breve{\phi} = UL(m\,m^T) - \int_\Delta m\,m^T \, d\mu_\Delta,
\]
where $U$ and $L$ are the Koopman and transfer operators for $f$.

There are a number of limit laws in the literature that require control of
Birkhoff sums corresponding to $\breve{\phi}$.
As examples, we mention the martingale CLT/WIP~\cite{Billingsley99} (see Appendix~\ref{sec-mda}) and an ASIP for reverse martingale differences~\cite{CunyMerlevede15} discussed at the end of this section.  
A third application~\cite{AntoniouMprep}
 is to the estimate of convergence rates in the WIP.
To control the Birkhoff sums of $\breve{\phi}$, a martingale-coboundary decomposition for $\breve{\phi}$
is of great utility; this is the topic of the current section.

\begin{prop}
  \label{prop-iLip2}
$|\breve{\phi}|_\infty \le \|v\|_\eta^2$ and
  \(\|P\breve{\phi}'\|_\eta \leq C \|v\|_\eta^2\).
\end{prop}

\begin{proof}
Let $U_F$ and $P$ denote the Koopman and transfer operators on $L^1(Y)$ for $F$
(so $P$ is as in Subsection~\ref{subsec-NUE} and $U_Fv=v\circ F$).  
A calculation shows that
\[
(L(mm^T))(y,\ell)=\begin{cases} (P(m'm'^T))(y) & \ell=0
\\ 0 & \ell\ge1 \end{cases},
\]
and hence that
\[
(UL(mm^T))(y,\ell)=\begin{cases} 0 & \ell\le\tau(y)-2 \\ (U_FP(m'm'^T))(y) & \ell = \tau(y)-1 \end{cases}.
\]
  
By Proposition~\ref{prop-mr}, $|m|_2  \ll \|v\|_\eta$.
  Also, 
  $\|\chi'\|_\eta \ll \|v\|_\eta$ so
$|m'| \le |\phi'|+2|\chi'|_\infty\ll \tau\|v\|_\eta$.
It follows that 
$|P(m'm'^T)(y)|\le \sum_{a\in\alpha}\zeta(y_a)|m'm'^T(y_a)|
\ll \sum_{a\in\alpha}\mu_Y(a)\tau(a)^2\|v\|_\eta^2\ll \|v\|_\eta^2$.
Hence
\[ \SMALL
|\breve{\phi}|_\infty \le |UL(mm^T)|_\infty+
|\int_\Delta mm^T\,d\mu_\Delta|
\le |P(m'm'^T)|_\infty+|m|_2^2\ll \|v\|_\eta^2.
\]

  It remains to estimate \(\|P \breve{\phi}'\|_\eta \).  Now
$\breve{\phi}'(y) = \sum_{\ell=0}^{\tau(y)-1} \breve{\phi}(y,\ell)
=(U_FP(m'm'^T))(y)-\tau(y)\int_\Delta m\,m^T\,d\mu_\Delta$, so
\[
P\breve{\phi}'=P(m'm'^T)-(P\tau)|m|_2^2.
\]
In particular, $|P \breve{\phi}'|_\infty \le |P(m'm'^T)|_\infty+|P\tau|_\infty|m|_2^2
\ll \|v\|_\eta^2$.

  Let \(x,y \in Y\) and \(a \in \alpha\).
  Using equation \eqref{eq-sbyz},
  \begin{align*}
    | m'(x_a) - m'(y_a) | 
    &  \leq |\phi'(x_a) - \phi'(y_a)| + |\chi'(x) - \chi'(y)| 
    + |\chi'(x_a) - \chi'(y_a)|
    \\ & \ll \tau(a) \|v\|_\eta d_\Lambda(x,y)^\eta,
  \end{align*}
  and so
  \begin{align*}
      \big|m'(x_a) m'(x_a)^T - m'(y_a) m'(y_a)^T \big|
     & \leq  \big(|m'(x_a)|+ |m'(y_a)|\big) | m'(x_a) - m'(y_a) |
    \\ & \ll  \tau(a)^2\|v\|_\eta^2  d_\Lambda(x,y)^\eta.
  \end{align*}
  As in the proof of Proposition~\ref{prop-iLip}, 
  \begin{align*}
    |(P & (m'm'^T))(x) -  (P (m'm'^T))(y)|
     \\ & \le  \sum_{a \in \alpha} 
      |\zeta(x_a) - \zeta(y_a)||(m'm'^T)(x_a)|+
    \sum_{a \in \alpha} 
       \zeta(y_a)|(m'm'^T)(x_a)-(m'm'^T)(y_a)|
    \\ & \ll \|v\|_\eta^2 \sum_{a \in \alpha} 
    \mu_Y(a) \tau(a)^2 d_\Lambda(x,y)^\eta 
    \ll \|v\|_\eta^2 d_\Lambda(x,y)^\eta.
  \end{align*}  
Hence $|P(m'm'^T)|_\eta \ll \|v\|_\eta^2$.
A simpler computation shows that $|P\tau|_\eta\ll 1$.
It follows that $|P\breve{\phi}'|_\eta \ll \|v\|_\eta^2$,
and so $\|P\breve{\phi}'\|_\eta \ll \|v\|_\eta^2$.
\end{proof}

Proposition~\ref{prop-iLip2} together with Remark~\ref{rmk-phi}
allows us to write
\(\breve{\phi} = \breve{m} + \breve{\chi} \circ f - \breve{\chi}\),
as in Subsection~\ref{sec-primary}.
In particular,
\begin{align} \label{eq-second} \nonumber
& |\breve{m}|_p\le C \|v\|_\eta^2, \quad 
\max_{\,0\le k\le n}|\breve{\chi}\circ f^k|=o(n^{1/p})\;a.e., \\
  & \big| \max_{1 \leq k \leq n} |\breve{\chi} \circ f^k - \breve{\chi}| \big|_p 
  \leq C \|v\|_\eta^2\, n^{1/p}\; \text{for all $n\ge1$}.
\end{align}
We refer to \(\breve{\phi} = \breve{m} + \breve{\chi} \circ f - \breve{\chi}\)
as a {\em secondary martingale-coboundary decomposition}.

\begin{cor}
  \label{cor-second} 
    $\big| \max_{1\le k\le n}|\sum_{j=0}^{k-1}\breve{\phi}\circ f^j|\big|_p
            \leq C \|v\|_\eta^2\,  n^{1/2}$.
\end{cor}

\begin{proof}
This follows from the argument for Corollary~\ref{cor-moment}.
\end{proof}

\begin{rmk} \label{rmk-second}
In a previous version of this paper, we obtained a similar decomposition for
$\breve{\phi}_1=m\,m^T-\int_\Delta m\,m^T\,d\mu_\Delta$.
The main difference is that $\breve{\phi}'_1=m'm'^T-\tau\int_\Delta m\,m^T\,d\mu_\Delta$.
The estimate for $\|P\breve{\phi}'_1\|_\eta$ is unchanged.
However, we obtain the inferior
estimate 
$|\breve{m}|_{p/2}\ll \|v\|_\eta^2$, resulting in a weaker estimate in
Corollary~\ref{cor-second}.

Since probabilistic results in the literature are typically stated in terms of the 
conditional variances $ULmm^T=\E(mm^T|f^{-1}\cM)$ (where $\E$ and $\cM$ are as in Proposition~\ref{prop-mart}),  we have chosen to omit the decomposition for $\breve{\phi}_1$ in this paper.
\end{rmk}

\begin{cor}[ASIP] \label{cor-ASIP}
Suppose that $d=1$.   
Define $\sigma^2=\lim_{n\to\infty}n^{-1}\int_\Lambda(\sum_{j=0}^{n-1}v\circ T^j)^2\,d\mu$ and suppose that $\sigma^2>0$.  Then there exists a probability space $\Omega$ supporting a sequence of random variables $\{S_n\}$ with the same joint distributions
as $\{\sum_{j=0}^{n-1}v\circ f^j\}$
and a sequence $\{Z_n\}$ of i.i.d.\ random variables with distribution $N(0,\sigma^2)$ such that almost everywhere as $n\to\infty$
\[
\SMALL \sup_{1\le k\le n} \big|S_k-\sum_{j=1}^k Z_j\big|=
\begin{cases} 
o\big((n\log\log n)^{1/2}\big)  & p=2, \\
o\big(n^{1/p}(\log n)^{1/2}\big)  & p\in(2,4), \\
O\big(n^{1/4}(\log n)^{1/2}(\log\log n)^{1/4}\big)  & p\ge4 
\end{cases}
\]
\end{cor}

\begin{proof}
Since $m\in\ker L$, it follows as in Proposition~\ref{prop-mart}
that $\E(m\circ f^n|f^{-(n+1)}\cM)= \E(m|f^{-1}\cM)\circ f^n =0$ for all $n\ge0$.
That is, $\{m\circ f^n\}$ is a sequence of {\em reverse martingale differences}
with respect to the nonincreasing sequence $\{f^{-n}\cM\}$ of $\sigma$-algebras.

We apply results of~\cite{CunyMerlevede15} to deduce the conclusion of the corollary with
the sequence $\{\sum_{j=0}^{n-1}v\circ T^j\}$ replaced by 
the sequence $\{\sum_{j=0}^{n-1}m\circ f^j\}$.
Suppose that this is the case.
By Proposition~\ref{prop-chiae}, 
\[
\SMALL \max_{1\le k\le n}\big|\sum_{j=0}^{k-1}(\phi\circ f^j-m\circ f^j)\big|
\le \max_{1\le k\le n}|\chi\circ f^k-\chi|=o(n^{1/p})\;a.e.
\]
Enlarging the probability space $\Omega$
(cf.~\cite[p.~23]{PhilippStout75}), there exists a sequence $\{S_n'\}$ with the same joint distributions as $\{\sum_{j=0}^{n-1}\phi\circ f^j\}$ so that
$\SMALL \sup_{1\le k\le n} |S_k'-\sum_{j=1}^k Z_j|$
satisfies the desired estimates.
Finally, $\pi_\Delta$ is a measure-preserving semiconjugacy, so the joint distributions of $\{\sum_{j=0}^{n-1}v\circ T^j\}$ also coincide with those of $\{S_n'\}$.

It remains to prove the ASIP (with the appropriate error rates) for the sequence $\{\sum_{j=0}^{n-1}m\circ f^j\}$.
The case $p=2$ is immediate from~\cite[Corollary~2.5]{CunyMerlevede15}.

When $p>2$, we require almost sure estimates for the sequence
$A_n=\sum_{j=0}^{n-1}(\E(m^2\circ f^j|\cG_{j+1})-\sigma^2)$.  For this 
we use the secondary martingale-coboundary decomposition 
$ULm^2-\sigma^2 =\breve{m}+\breve{\chi}\circ f-\breve{\chi}$.
Then
\[
\SMALL A_n =
\sum_{j=0}^{n-1}((ULm^2)\circ f^j-\sigma^2)
=\sum_{j=0}^{n-1}\breve{m}\circ f^j + \breve{\chi}\circ f^n-\breve{\chi}.
\]

Since $\{\breve{m}\circ f^n\}$ is a sequence of $L^2$ reverse martingale differences, the result for $p=2$ implies an ASIP for $\{\sum_{j=0}^{n-1}\breve{m}\circ f^j\}$ with error rate
$o((n\log\log n)^{1/2})$.   This is sufficient to deduce the law of the iterated logarithm $\sum_{j=0}^{n-1}\breve{m}\circ f^j=O((n\log\log n)^{1/2})$ a.e.
Also $\breve{\chi}\circ f^n-\breve{\chi}=o(n^{1/p})$ a.e., so 
$A_n = O((n\log\log n)^{1/2})$ a.e.

For $p\in(2,4)$, the desired ASIP for $m$ now follows 
from~\cite[Corollary~2.7]{CunyMerlevede15} (taking $b(n)\equiv1$).
For $p\ge4$, the desired ASIP for $m$ follows 
from~\cite[Corollary~2.8]{CunyMerlevede15}.
\end{proof}

For the class of (Markovian) nonuniformly expanding maps as defined in Section~\ref{subsec-NUE}, with H\"older observables $v$, our results improve existing results in the literature.   The best previous result that we are aware of
is~\cite[Theorem~3.5]{CunyMerlevede15} who obtained the error rate $O(n^{1/4}(\log n)^{1/2}(\log\log n)^{1/4})$ for $p>6$ (this constraint is required to ensure that $|L^nv|_4$ decays faster than $n^{-5/4}$ for mean zero H\"older $v$ so that
condition~(3.2) in~\cite{CunyMerlevede15} is satisfied), whereas we require only that $p\ge4$.  
(For one-dimensional dynamical systems and certain classes of (unbounded) observables,~\cite{CunyMerlevede15} obtain much better results.)

\begin{examp} \label{ex-LSV}
Consider the map $T:[0,1]\to[0,1]$ of intermittent type~\cite{PomeauManneville80} studied by~\cite{LiveraniSaussolVaienti99}, namely
$T(x)=\begin{cases} x(1+2^\gamma x^\gamma) & x\in[0,\frac12)
\\ 2x-1 & x\in[\frac12,1] \end{cases}$.
It is standard that $T$ is a nonuniformly expanding map with absolutely continuous invariant probability measure $\mu$ for each $\gamma\in(0,1)$.
The inducing time $\tau$
lies in
$L^p$ if and only if $p<1/\gamma$.
Hence, we obtain the error rate $O(n^{1/4}(\log n)^{1/2}(\log\log n)^{1/4})$
for all mean zero H\"older observables~$v$ provided $\gamma<\frac14$; previously this was known only for $\gamma<\frac16$.
\end{examp}

\begin{examp}
We consider a
family of planar periodic dispersing billiards introduced by~\cite{ChernovZhang05b}.  The scatterers have smooth strictly convex boundaries with nonvanishing curvature, except that the curvature vanishes at two points.   Moreover, it is assumed that there is a periodic orbit that runs between the two flat points, and that the boundary near these flat points has the form $\pm(1+|x|^b)$ for some $b>2$.  

By~\cite{ChernovZhang05b}, quotienting out stable manifolds leads to a nonuniformly expanding map $T$ with inducing time $\tau$ lying in $L^p$
for all $p<(b+2)/(b-2)$.   Hence at the level of the quotient map we obtain the error rate $O(n^{1/4}(\log n)^{1/2}(\log\log n)^{1/4})$
for all mean zero H\"older observables $v$ provided $b<\frac{10}{3}$; previous results require $b<\frac{14}{5}$.  We conjecture that these results go over to the full (unquotiented map); this is the topic of future work.
\end{examp}

\begin{examp}
Bunimovich flowers~\cite{Bunimovich73} are billiards
where the boundary components of the billiard table are either dispersing, or focusing arcs of circles, subject to some technical constraints.
By~\cite{ChernovZhang05}, the quotient map $T$ (obtained by quotienting out stable manifolds) is nonuniformly expanding with inducing time $\tau\in L^p$ for all $p<3$.
Hence, at least at the level of the quotient map we obtain the error rate $o(n^{1/q})$ for all $q<3$.
\end{examp}

\section{Limit laws for families of nonuniformly expanding maps}
\label{sec-NUElimit}

In this section, we show how the martingale-boundary decompositions from the previous sections apply to Birkhoff sums of the type
$\sum_{j=0}^{n-1}v_n\circ T_n^j$ where the dynamical systems $T_n$ vary with $n$.

Suppose that $T_n:\Lambda_n\to\Lambda_n$, $n\ge0$, is a family of nonuniformly expanding maps
as defined in Section~\ref{subsec-NUE}, with absolutely continuous
ergodic $T_n$-invariant probability measures $\mu_n$.
Let $\tau_n:Y_n\to\Z^+$ and $F_n:Y_n\to Y_n$ be the corresponding inducing times and induced maps with ergodic $F_n$-invariant probability measures $\mu_{Y_n}$.
We say that $T_n:\Lambda_n\to\Lambda_n$ is a {\em uniform family of order $p\ge1$} if
\begin{itemize}

\parskip=-2pt
\item[(i)] $\sup_{n\ge0}\diam\Lambda_n<\infty$ and the constants $C_0,C_1\ge1$, $\lambda>1$, $\eta\in(0,1]$
can be chosen independent of $n\ge0$.
\item[(ii)]   The family $\{\tau_n^p,\, n\ge0\}$ is uniformly integrable.
(For this it suffices that
        $\sup_{n\ge0}\int_{Y_n}\tau_n^q\,d\rho<\infty$ for
some $q>p$.)
\end{itemize}

Let $v_n:\Lambda_n\to\R^d$ be a family of H\"older observables with
$\int_{\Lambda_n} v_n\,d\mu_n=0$.  We suppose that
$\sup_{n\ge0}\|v_n\|_\eta<\infty$.

Let $\Delta_n$ be the corresponding family of Young towers defined as in Section~\ref{subsec-NUE}, with maps $f_n:\Delta_n\to\Delta_n$, invariant probability measures $\mu_{\Delta_n}$ and semiconjugacies
$\pi_{\Delta_n}:\Delta_n\to \Lambda_n$.

Let $\phi_n=v_n\circ \pi_{\Delta_n}:\Delta_n\to\R^d$.
By the results from Section~\ref{sec-primary}, we have the primary martingale-coboundary decomposition
\begin{align} \label{eq-prim}
v_n\circ \pi_{\Delta_n}=\phi_n= m_n+\chi_n\circ f_n-\chi_n,
\end{align}
where $m_n,\,\chi_n$ satisfy the estimates in Propositions~\ref{prop-mr} and~\ref{prop-chi}  uniformly in $n$.

\begin{lemma} \label{lem-moment}  
If $p\le 2$, then
$\big|\max_{j\le n}|\sum_{k=0}^{j-1}v_n\circ T_n^k|\big|_p\le C\|v_n\|_\eta n^{1/p}$
and
$\big|\max_{j\le n}|\sum_{k=0}^{j-1}m_n\circ f_n^k|\big|_p\le C\|v_n\|_\eta n^{1/p}$
for all \mbox{$n\ge0$}.

If $p\ge 2$, then
$\big|\max_{j\le n}|\sum_{k=0}^{j-1}v_n\circ T_n^k|\big|_{2(p-1)}\le C\|v_n\|_\eta n^{1/2}$
for all \mbox{$n\ge0$}.
\end{lemma}

\begin{proof}
This is immediate from Corollary~\ref{cor-moment} since the constant $C$ there is 
independent of $n$.~
\end{proof}

From now on, we suppose that $p\ge2$.
By Corollary~\ref{cor-Sigma}, we can define the family of covariance matrices
\begin{align}   \label{eq-Sigma}
\Sigma_n
& =\lim_{k\to\infty}k^{-1}\int_{\Lambda_n} 
\Big(\sum_{j=0}^{k-1}v_n\circ T_n^j\Big)
\Big(\sum_{j=0}^{k-1}v_n\circ T_n^j\Big)^T\,d\mu_n
= \int_{\Delta_n} m_n\,m_n^T\,d\mu_{\Delta_n}.
\end{align}

\begin{rmk} \label{rmk-unifSigma}
It follows from the proof of Corollary~\ref{cor-Sigma} that 
the convergence in~\eqref{eq-Sigma} is uniform in $n$.
\end{rmk}

By the results from Section~\ref{sec-secondary}, we have the secondary martingale 
coboundary decomposition
\begin{align} \label{eq-sec}
U_nL_n(m_nm_n^T)- \Sigma_n
=\breve{\phi}_n = \breve{m}_n+\breve{\chi}_n\circ f_n-\breve{\chi}_n,
\end{align}
where $U_n$ and $L_n$ are the Koopman and transfer operators for $f_n$.

\begin{prop} \label{prop-UI}
The family $\{|m_n|^2,\, n\ge0\}$ is uniformly integrable.
\end{prop}

\begin{proof}
We start from the primary decomposition~\eqref{eq-prim}, with
$\phi_n'=m_n'+\chi_n'\circ F_n-\chi_n'$ (as in Section~\ref{sec-primary}).

Since $|v_n|_\infty$ is bounded, it is immediate from condition~(ii) and
the definition
$\phi_n'=\sum_{j=0}^{\tau-1}v_n\circ \pi_{\Delta_n}\circ f_n^j$ that
$\{|\phi_n'|^2\}$ is uniformly integrable.
Next,
$|\chi_n'|_\infty$ is bounded and hence $\{|m_n'|^2\}$ is uniformly integrable.
It follows from the proof of Proposition~\ref{prop-mr} that
the uniform integrability of $\{|m_n'|^2\}$ is inherited by $\{|m_n|^2\}$.
\end{proof}

Let $W_n(t)=n^{-1/2}\sum_{j=0}^{[nt]-1}v_n\circ T_n^j$.

\begin{prop} \label{prop-limit}
Suppose that $\lim_{n\to\infty}\Sigma_n=\Sigma$ where $\Sigma\in\R^{d\times d}$.
Then $W_n\to_{\mu_n} W$ in $D([0,\infty),\R^d)$
where $W$ is Brownian motion with covariance $\Sigma$.
\end{prop}

\begin{proof}
Define processes 
$\Phi_n(t)=n^{-1/2}\sum_{j=0}^{[nt]-1}\phi_n\circ f_n^j$,
$M_n(t)=n^{-1/2}\sum_{j=0}^{[nt]-1}m_n\circ f_n^j$.

By Proposition~\ref{prop-UI}, the family $\{|m_n|^2,\,n\ge0\}$ is uniformly integrable.
Next,
  \begin{align*}
n^{-1}\sum_{j=0}^{[nt]-1}\{U_nL_n & (m_nm_n^T)\}  \circ f_n^j  \,  -\,t\Sigma
  = 
n^{-1}\sum_{j=0}^{[nt]-1}\breve{\phi}_n\circ f_n^j
+ n^{-1}[nt]\Sigma_n-t\Sigma \to_{\mu_{\Delta_n}}0,
\end{align*}
by Corollary~\ref{cor-second}.
By Theorem~\ref{thm-mda}, $M_n\to_{\mu_{\Delta_n}} W$.

Let $T>0$.  
Since $T_n$ is a uniform family, 
$\big|\max_{1\le k\le nT}|\chi_n\circ f_n^k-\chi_n|\big|_2=o(n^{1/2})$
by Proposition~\ref{prop-chi}.
Also,
\begin{align*}
\SMALL\sup_{t\in[0,T]}|\Phi_n(t)-M_n(t)| & \SMALL\le
n^{-1/2}\max_{1\le k\le nT}|\sum_{j=0}^{k-1}(\phi_n\circ f_n^j-m_n\circ f_n^j)|
\\ &\SMALL \le 
n^{-1/2}\max_{1\le k\le nT}|\chi_n\circ f_n^k-\chi_n|.
\end{align*}
Hence
$\lim_{n\to\infty}\big|\sup_{t\in[0,T]}|\Phi_n(t)-M_n(t)|\big|_2=0$
for each $T>0$. It follows that
\mbox{$\Phi_n\to_{\mu_{\Delta_n}} W$}.
Also, $\pi_{\Delta_n}$ is a measure-preserving semiconjugacy for each $n$, 
so 
$W_n\to_{\mu_n}W$.
\end{proof}

Define $\cW\subset D([0,\infty),\R^d)$ to be the set of weak limits of $\{W_n,\,n\ge0\}$ and let
$\cS\subset\R^{d\times d}$ be the set of limit points of $\{\Sigma_n,\,n\ge0\}$.
By Proposition~\ref{prop-UI}, $\{\Sigma_n,\,n\ge0\}$ is bounded and hence
$\cS\neq\emptyset$.

\begin{thm}
  \label{thm-WIP}
(i) $\{W_n,\,n\ge0\}$ is tight, (ii)
$W\in \cW$ if and only if $W$ is a Brownian motion 
with covariance matrix in $\cS$.
\end{thm}

\begin{proof}
Given a covariance matrix $\Sigma\in\R^{d\times d}$,
let $W(\Sigma)$ denote Brownian motion with covariance $\Sigma$.

Since $\{\Sigma_n\}$ is bounded,
for any subsequence $W_{n_k}$ we can pass to a subsubsequence along which 
$\Sigma_{n_k}\to\Sigma$ for some $\Sigma\in\cS$.
By Proposition~\ref{prop-limit}, we then have that $W_{n_k}\to_{\mu_{n_k}}W(\Sigma)$.  This shows that $\{W_n\}$ is tight and that all weak limits
have the form $W(\Sigma)$, $\Sigma\in\cS$.

Conversely, 
if $\lim_{k\to\infty}\Sigma_{n_k}=\Sigma$ for some subsequence $n_k$,
then $W_{n_k}\to_{\mu_n}W(\Sigma)$ by Proposition~\ref{prop-limit}.
\end{proof}

\begin{cor}  Suppose in Theorem~\ref{thm-WIP} that $W_{n_k}\to_{\mu_{n_k}}W$ as $k\to\infty$.  Then
$\lim_{k\to\infty}n_k^{-q/2}\int_{\Lambda_{n_k}}|\sum_{j=0}^{n_k-1}v_{n_k}\circ T_{n_k}^j|^q\,d\mu_{n_k}=\E|W(1)|^q$ for all $q<2(p-1)$.
\end{cor}

\begin{proof}  This follows immediately from Lemma~\ref{lem-moment} (cf.~\cite[Lemma~2.1(e)]{MTorok12}).
\end{proof}

\begin{rmk}  It is not difficult to formulate conditions under which the weak limits of $\{W_n\}$ are nondegenerate.
One possibility is to suppose that there is a limiting nonuniformly expanding map $T_\infty$ with corresponding observable $v_\infty$ and covariance matrix $\Sigma_\infty$.  Typically $\det\Sigma_\infty>0$.  Under certain conditions (see for example Section~\ref{sec-homog-uniform}), it can be shown that $W_n\to_{\mu_n}W$ where $W$ is Brownian motion with covariance $\Sigma_\infty$.

An alternative mechanism for nondegenerate limits is the following.  
Suppose that $d=1$.
Recall that 
$|\chi'_n|_2\le|\chi'_n|_\infty\le 
C\|v_n\|_\eta\le 
C\|v_n\|_\eta|\tau_n|_1$,
where $C>0$ is a constant depending only on the induced maps $F_n$.  Hence
\[
\sigma_n=|m_n|_2=|m_n'|_2/|\tau_n|_1\ge (|\phi'_n|_2-2|\chi'_n|_2)/|\tau_n|_1
\ge |\phi_n'|_2/|\tau_n|_1\,-\,C\|v_n\|_\eta.
\]

If we arrange that $|\tau_n|_2\ge 2C|\tau_n|_1$ for all $n$, then it follows that
\[
\sigma_n\ge C(2|\phi'_n|_2/|\tau_n|_2\,-\,\|v_n\|_\eta).
\]

Suppose for simplicity that $F_n=T_n^{\tau_n}$ is the first return map for each $n$.   
Let $k_n$ be largest such that $\sum_{j=1}^{k_n} j\mu_{Y_n}(\tau_n=j)>\sum_{j=k_n+1}^\infty j\mu_{Y_n}(\tau_n=j)$,
There is a unique observable $v_n:\Lambda_n\to\R$ taking values $\pm1$ such that $\phi'_n(y)=\tau_n(y)$ if $\tau_n(y)\le k_n$ and 
$\phi'_n(y)=-\tau_n(y)$ if $\tau_n(y)> k_n$.
By construction $\int_{Y_n}\phi_n'\,d\mu_{Y_n}\approx0$ and $|\phi_n'|_2=|\tau_n|_2$.
A slight modification produces $\int_{Y_n}\phi_n'\,d\mu_{Y_n}=0$ and $|\phi_n'|_2\approx|\tau_n|_2$ so that $\sigma_n\gtrsim C$.
Hence this is a robust mechanism for producing nondegenerate limits in Theorem~\ref{thm-WIP}.
  \end{rmk}

We end this section with some examples of nonuniformly expanding maps where uniformity of the constants can be verified, and hence to which the results in this section apply.

\begin{examp}
Fix a sequence
$\lambda_n\in\Z$ such that $\lambda_n\ge2$ for all $n\ge0$.
Define the family of uniformly expanding maps $T_n:[0,1]\to[0,1]$ given by
$T_nx=\lambda_nx\bmod1$.   Clearly $T_n$ is a uniform family of order $p$ for any $p$, with $Y_n=[0,1]$, $\tau_n=1$ and $\mu_n={\rm Lebesgue}$.

This example emphasizes that tightness in Theorem~\ref{thm-WIP} is unrelated to any accumulation properties of the dynamical systems $T_n$ and is governed purely by accumulation of the bounded set of covariance matrices.
\end{examp}

\begin{examp} \label{ex-LSV2}
Fix a sequence $\gamma_n\in(0,1)$ and let $T_n:[0,1]\to[0,1]$ be the corresponding intermittent map defined in Example~\ref{ex-LSV}.
As verified in~\cite[Example~5.1]{KKMsub}, $T_n$ is a uniform family of order
$p$ for any $p<\sup \gamma_n^{-1}$.
\end{examp}

\begin{examp} \label{ex-CE}
Consider the family of quadratic maps $T_n:[-1,1]\to[-1,1]$ given by
$T_n(x)=1-a_nx^2$, $a_n\in[0,2]$.
We assume that there exists $b,c>0$ such that
the Collet-Eckmann condition~\cite{ColletEckmann83} 
$|(T_n^k)'(1)|\ge ce^{bn}$ holds for all $k$, $n\ge0$.
(By~\cite{Jakobson81,BenedicksCarleson85}, the set of parameters $a_n$ for which this condition holds has positive Lebesgue measure for $b,c$ sufficiently small.)
As verified in~\cite[Example~5.2]{KKMsub} (based on arguments of~\cite{FreitasTodd09}),
$T_n$ is a uniform family of order $p$ for any $p$.  This example generalises to multimodal maps
(see~\cite[Example~5.3]{KKMsub}.
\end{examp}

\begin{examp} \label{ex-Viana}
Viana~\cite{Viana97} introduced a $C^3$ open class of multi-dimensional nonuniformly expanding maps $T_\eps:M\to M$.  For definiteness, we restrict attention to the case $M=S^1\times\R$.
Fix $\lambda_n\in\Z$, $\lambda_n\ge16$, and let
$S_n:M\to M$ be the map
$S_n(\theta,y)=(\lambda_n\theta\bmod1,a_0+a\sin 2\pi\theta-y^2)$.
Here $a_0$ is chosen so that $0$ is a preperiodic point for the quadratic map $y\mapsto a_0-y^2$ and $a$ is fixed sufficiently small.
Let
$T_n$ be a family of $C^3$ maps each of which is sufficiently close to $S_n$.
It follows from~\cite{Alves00,AlvesViana02} that there is an interval $I\subset(-2,2)$ such that, for each $n\ge0$, there is a unique absolutely continuous $T_n$-invariant ergodic probability measure $\mu_n$ supported in the interior of $S^1\times I$.  Moreover the invariant set $\Lambda_\eps=\supp\mu_\eps$ attracts almost every initial condition in $S^1\times I$.

As verified in~\cite[Example~5.4]{KKMsub} (based on arguments of~\cite{Alves04,AlvesLuzzattoPinheiro05}), $T_n:\Lambda_n\to\Lambda_n$ is a uniform family of nonuniformly expanding maps
of order $p$ for any $p$.
\end{examp}

\section{Limit laws for families of nonuniformly hyperbolic transformations}
\label{sec-NUHlimit}

In this section, we show the results from Section~\ref{sec-NUElimit} pass over to the invertible setting.
The notion of nonuniformly hyperbolic transformation is recalled in
Subsection~\ref{sec-NUH}.
In Subsection~\ref{sec-quotient}, we recall how to quotient to a nonuniformly expanding map.  In Subsection~\ref{subsec-NUHlimit}, we prove limit laws for families of nonuniformly hyperbolic transformations.

\subsection{Nonuniformly hyperbolic transformations}
\label{sec-NUH}

Let $T:\Lambda \to\Lambda$ be a diffeomorphism (possibly with singularities) defined on a Riemannian manifold $(\Lambda,d_\Lambda)$.  We assume that $T$ is nonuniformly hyperbolic in the sense of Young~\cite{Young98,Young99}.  The precise definitions are somewhat technical; here we are content to focus on the parts necessary for understanding this paper, referring to~\cite{Young98,Young99} for further details.

As part of this set up, there is a measurable (with respect to the Riemannian measure) set $Y\subset M$, a measurable partition $\{Y_j\}$ of $Y$, and an inducing time
$\tau: Y \to \Z^+$ constant on partition elements such that $T^{\tau(y)}(y)\in Y$
for all $y\in Y$.  We refer to $F=T^\tau:Y\to Y$ as the induced map.
The separation time $s(y,y')$ of points $y,y'\in Y$ is the least integer $n\ge0$ such that $F^ny,F^ny'$ lie in distinct partition elements of $Y$.

In addition, there exist integers $d_s,\,d_u\ge1$ with $d_s+d_u=\dim M$, a measurable partition $\cW^s$ of $Y$ consisting of embedded $d_s$-dimensional disks (called ``stable leaves'') and an embedded $d_u$-dimensional disk $W^u$ (called an ``unstable leaf'') such that
$W^u$ intersects each element of $\cW^s$ in a single point.
If $y\in Y$, the leaf in $\cW^s$ that contains $y$ is labelled $W^s_y$.
Let $\rho$ denote the measure on $W^u$ induced by the Riemannian measure.

We assume that there are constants $D_0$, $D_1\ge 1$, $\gamma\in(0,1)$, $p\ge1$, such that
\begin{itemize}
\item[(A1)] Each $Y_j$ is a union of elements of $\cW^s$
(in particular, $\tau$ is constant on stable leaves),  
 and $F(W^s_y)\subset W^s_{Fy}$ for all $y\in Y$.
\item[(A2)] 
\begin{itemize}
\item[(i)]
 $d_\Lambda(T^jy,T^jy')\le D_0\gamma^j$ for all $y\in Y$, $y'\in W^s_y$, 
\item[(ii)]
$d_\Lambda(T^j y,T^j y')\le D_0\gamma^{s(y,y')-\psi_j(y)}$ for all $y,y'\in W^u$,
\end{itemize}
for all $j\ge0$, where $\psi_j(y)=\#\{k=0,\dots,j-1:T^ky\in Y\}$ is the number of visits of $y$ to $Y$ by time $j$.
\item[(A3)] $\int_Y \tau^p\,d\rho<\infty$.
\end{itemize}

Let $\bar Y=Y/\sim$ where $y\sim y'$ if $y'\in W^s_y$, and let $\bar\pi:Y\to\bar Y$
denote the natural projection.  By (A1), we obtain well-defined functions
$\tau:\bar Y\to\Z^+$ and $\bar F:\bar Y\to\bar Y$.
Let $\bar\rho=\bar\pi_*\rho$.
Let $\alpha$ be the countable partition of $\bar Y$ consisting of the partition elements $Y_j$ quotiented by $W^s$.
We assume:
\begin{itemize}
\item[(A4)]
$\bar F$ restricts to a bijection from $a$ onto $\bar Y$ for all $a\in\alpha$
and $\zeta_0=\frac{d\bar\rho}{d\bar\rho\circ \bar F}$
satisfies $|\log \zeta_0(y)-\log \zeta_0(y')|\le D_1\gamma^{s(y,y')}$ for all
\mbox{$y,y'\in a$}.
\end{itemize}
There is a unique absolutely continuous $\bar F$-invariant probability measure $\bar\mu_Y$ on $\bar Y$
and $d\bar\mu_Y/d\bar\rho\in L^\infty$.
By for instance~\cite[Section~6.1]{APPV09},
there is a unique ergodic $F$-invariant 
probability measure $\mu_Y$ on $Y$ such that $\bar\pi_*\mu_Y=\bar\mu_Y$.

As in Section~\ref{sec-NUElimit}, we define a tower map $f:\Delta\to\Delta$
with semiconjugacy $\pi_\Delta:\Delta\to\Lambda$ from $f$ to $T$, and ergodic $f$-invariant probability measure $\mu_\Delta=\mu_Y\times{\rm counting}/\int_Y\tau\,d\mu_Y$.
Then
$\mu=(\pi_\Delta)_*\mu_\Delta$ is an ergodic $T$-invariant probability measure on $M$.

\begin{rmk} \label{rmk-Ws}
For simplicity, we restrict to the case where $T$ contracts exponentially along stable manifolds.  It is also possible to consider polynomial (but summable) contraction as in~\cite{AlvesAzevedo16}, as well as the general situation~\cite{MV16} where contraction and expansion is assumed only on returns to $Y$.  (The arguments to treat this general situation are correspondingly longer.)
\end{rmk}

Next, we introduce the quotient
tower map $\bar f :\bar\Delta\to\bar\Delta$ defined in the same way as $f:\Delta\to\Delta$ but starting from $\bar F:\bar Y\to \bar Y$ instead of $F:Y\to Y$.
The projection $\bar\pi:Y\to\bar Y$ extends to a projection $\bar\pi:\Delta\to\bar\Delta$, $\bar\pi(y,\ell)=(\bar\pi y,\ell)$, and we have the ergodic $\bar f$-invariant probability measure $\bar\mu_\Delta=\bar\pi_*\mu_\Delta=\bar\mu_Y\times{\rm counting}/\int_{\bar Y}\tau\,d\bar\mu_Y$.

The separation time $s$ on $Y$ projects to a separation time on $\bar Y$.
For $\theta\in(0,1)$ we define the symbolic metric $d_\theta$ on $\bar Y$,
setting $d_\theta(y,y')=\theta^{s(y,y')}$.
This extends to a metric on $\bar\Delta$, where $d_\theta((y,\ell),(y',\ell'))=
\begin{cases} d_\theta(y,y') & \ell=\ell' \\ 1 & \ell\neq\ell'\end{cases}$.

\begin{prop} \label{prop-bar}
Choose $\theta\in[\gamma,1)$.  Then
$\bar f:\bar\Delta\to\bar\Delta$ is a nonuniformly expanding map on the metric
space $(\bar\Delta,d_\theta)$ with induced map $\bar F:\bar Y\to\bar Y$,
partition $\alpha$,
and constants  $\lambda>1$, $\eta\in(0,1]$, $C_0,C_1\ge1$ given by
$\lambda=\theta^{-1}$, $\eta=C_0=1$, $C_1=D_1$.
\end{prop}

\begin{proof}
By (A4), $\bar F$ maps partition elements bijectively onto $\bar Y$.
By definition of $d_\theta$, if $y,y'\in a$, $a\in\alpha$, then
$d_\theta(\bar Fy,\bar Fy')=\theta^{-1}d_\theta(y,y')$ 
and 
$d_\theta(\bar f^\ell y,\bar f^\ell y')=d_\theta(y,y')\le d_\theta(\bar Fy,\bar Fy')$ for all $0\le \ell<\tau(a)$.
Finally, by (A4),
$|\log \zeta_0(x)-\log \zeta_0(y)|\le D_1\gamma^{s(y,y')}\le D_1 d_\theta(\bar Fy,\bar Fy')$.~
\end{proof}

\subsection{Quotienting step}
\label{sec-quotient}

In this subsection, we recall a standard procedure for reducing limit laws for nonuniformly hyperbolic transformations to the noninvertible (nonuniformly expanding) setting.  We work throughout with H\"older observables $v\in C^\eta$, where $\eta\in(0,1]$ is fixed.

First, define a projection $Y\to W^u$ by setting
$\hat y=W^s_y\cap W^u$ for $y\in Y$.   This extends to a projection on $\Delta$ by setting $\hat p=(\hat y,\ell)$ for $p=(y,\ell)\in\Delta$.

Given $v:\Lambda\to\R^d$ H\"older, define $\psi:\Delta\to\R^d$,
\[
\SMALL\psi(p)=\sum_{j=0}^\infty \{v\circ\pi_\Delta(f^jp)-
v\circ\pi_\Delta(f^j\hat p)\}.
\]

\begin{prop} \label{prop-psi}
Let $\theta\in[\gamma^\eta,1)$.  Then
$|\psi|_\infty\le D_0^\eta(1-\theta)^{-1}|v|_\eta$, for all H\"older $v:\Lambda\to\R^d$.
\end{prop}

\begin{proof}
Let $p=(y,\ell)\in\Delta$. Then $\pi_\Delta(f^jp)=T^{j+\ell} y\in T^{j+\ell} W^s_y$
and $\pi_\Delta(f^j\hat p)=T^{j+\ell} \hat y\in T^{j+\ell} W^s_{\hat y}=T^{j+\ell} W^s_y$.
In particular, 
$d_\Lambda(\pi_\Delta(f^jp),\pi_\Delta(f^j\hat p))\le D_0\gamma^{j+\ell}\le D_0\gamma^j$ by (A2)(i).  Hence
$|\psi(p)|\le |v|_\eta\sum_{j=0}^\infty d_\Lambda(\pi_\Delta(f^jp),\pi_\Delta(f^j\hat p))^{\eta}\le D_0^{\eta}(1-\theta)^{-1}|v|_\eta$.
\end{proof}

A calculation shows that 
$v\circ\pi_\Delta=\hat v+\psi-\psi\circ f$,
where $\hat v\in L^\infty(\Delta)$ is given by
\[
\SMALL \hat v(p)=
v\circ\pi_\Delta( \hat p)+\sum_{j=0}^\infty \{v\circ\pi_\Delta(f^{j+1}\hat p)-
v\circ\pi_\Delta(f^j\widehat{fp})\}.
\]
Note that $\hat v$ is constant along fibres $\bar\pi^{-1}(y)$, $y\in\bar Y$.
Hence we can write $\hat v=\bar v\circ\bar\pi$ where $\bar v:\bar \Delta\to\R^d$.
The observables $v:\Lambda\to\R^d$ and $\bar v:\bar\Delta\to\R^d$ are related by the equation
\begin{align} \label{eq-vv}
v\circ\pi_\Delta=\bar v\circ\bar\pi+\psi-\psi\circ f.
\end{align}
Let $\|\bar v\|_\theta=|\bar v|_\infty+|\bar v|_\theta$ where $|\bar v|_\theta$ denotes the $d_\theta$-Lipschitz constant of $\bar v$.

\begin{prop}  \label{prop-vv}
Let $\theta\in[\gamma^{\eta/2},1)$.  Then
$\|\bar v\|_\theta\le 6D_0^\eta\theta^{-2}(1-\theta^2)^{-1}\|v\|_\eta$,
for all H\"older $v:\Lambda\to\R^d$.
\end{prop}

\begin{proof}
By Proposition~\ref{prop-psi}, $|\bar v|_\infty\le 2D_0^\eta(1-\theta)^{-1}\|v\|_\eta$.

Next, let $p=(y,\ell)$, $q=(z,\ell)\in\Delta$.
Write $|\hat v(p)-\hat v(q)|\le A_1+A_2+A_3+A_4$ where
\begin{alignat*}{2} 
A_1 & = \sum_{j=N}^\infty|v\circ\pi_\Delta(f^{j+1}\hat p)-v\circ\pi_\Delta(f^j\widehat{fp})|, 
& \quad
A_2 & = \sum_{j=N}^\infty|v\circ\pi_\Delta(f^{j+1}\hat q)-v\circ\pi_\Delta(f^j\widehat{fq})|, \\
A_3 & = \sum_{j=0}^N|v\circ\pi_\Delta(f^j\hat p)-v\circ\pi_\Delta(f^j\hat q)|,
& \quad
A_4 & = \sum_{j=0}^{N-1}|v\circ\pi_\Delta(f^j\widehat{fp})-v\circ\pi_\Delta(f^j\widehat{fq})|.
\end{alignat*}

We take $N=[\frac12 s(y,z)]$ and show that
$A_j\le 
D_0^\eta\theta^{-2}(1-\theta^2)^{-1}|v|_\eta\,\theta^{s(y,z)}$ for $j=1,\dots,4$.
Hence $|\hat v|_\theta\le 4D_0^\eta\theta^{-2}(1-\theta^2)^{-1}|v|_\eta d_\theta(p,q)$ and the proof is complete.

First, we estimate $A_1$.
If $\ell<\tau(y)-2$ then $f\hat p=\widehat{fp}=(\hat y,\ell+1)$ and \mbox{$A_1=0$}.
Otherwise $f\hat p=(F\hat y,0)$ and
$\widehat{fp}=(\widehat{Fy},0)$.
In particular, 
$\pi_\Delta(f\hat p),\, \pi_\Delta(\widehat{fp})\in W^s_{Fy}$, 
so it follows from (A2)(i) that
\[
d_\Lambda(\pi_\Delta(f^{j+1}\hat p),\pi_\Delta(f^j\widehat{fp}))=
d_\Lambda(
T^j(\pi_\Delta(f\hat p)), T^j(\pi_\Delta(\widehat{fp}))) \le D_0\gamma^j.
\]
Hence $A_1\le |v|_\eta D_0^\eta\sum_{j=N}^\infty \theta^{2j}\le 
D_0^\eta(1-\theta^2)^{-1}\theta^{2N}|v|_\eta
\le D_0^\eta\theta^{-2}(1-\theta^2)^{-1}|v|_\eta\,\theta^{s(y,z)}$.
Similarly $A_2\le 
 D_0^\eta\theta^{-2}(1-\theta^2)^{-1}|v|_\eta\,\theta^{s(y,z)}$.

Next, we estimate $A_3$.
For each $j$, we have $\pi_\Delta(f^j\hat p)=T^{j+\ell} \hat y=T^LF^J\hat y$
where $0\le J\le j$ and $0\le L<\tau(F^J\hat y)$.
Note that $F^jy$ and $F^jz$ lie in the same partition element of $Y$ for all $j\le N$, so $\pi_\Delta(f^j\hat q)=T^{j+\ell} \hat z=T^LF^J\hat z$.
By (A2)(ii), 
\[
d_\Lambda(\pi_\Delta(f^j\hat p),\pi_\Delta(f^j\hat q))= d_\Lambda(T^LF^J\hat y,T^LF^J\hat z)\le 
D_0\gamma^{s(y,z)-J} \le D_0\gamma^{s(y,z)-j}.
\]
Hence $A_3\le 
D_0^\eta(1-\theta^2)^{-1}|v|_\eta\theta^{2(s(y,z)-N)}
\le D_0^\eta(1-\theta^2)^{-1}|v|_\eta\theta^{s(y,z)}$.
Similarly, $A_4\le D_0^\eta(1-\theta^2)^{-1}|v|_\eta\theta^{s(y,z)}$.
\end{proof}

\begin{cor} \label{cor-UMESigma}
Suppose that $v:\Lambda\to\R^d$ is H\"older with $\int_\Lambda v\,d\mu=0$.
Let $\bar v:\bar\Delta\to\R^d$ be the corresponding $d_\theta$-Lipschitz observable.
Then
\begin{itemize}
\item[(a)]
$\big|\max_{j\le n}|\sum_{k=0}^{j-1} v\circ T^k|\big|_{L^{p^*}\!(\mu)}\le
C \|v\|_\eta\, n^{\max\{\frac12,\frac1p\}}$ for all $n\ge1$, where $C\ge1$ is a constant that depends continuously on $D_0,D_1,\gamma$, and $p^*=\max\{p,2(p-1)\}$.
\item[(b)] Suppose that $p\ge2$.  Then the limits
\[
\SMALL \Sigma=\lim_{n\to\infty}n^{-1}\int_\Lambda
S_nv \,S_nv^T\,d\mu
=\lim_{n\to\infty}n^{-1}\int_{\bar\Delta} 
S_n\bar v \,S_n\bar v^T
\,d\bar\mu_\Delta,
\]
exist and coincide, where
$S_nv=\sum_{j=0}^{n-1}v\circ T^j$,
$S_n\bar v=\sum_{j=0}^{n-1}\bar v\circ \bar f^j$.
\end{itemize}
\end{cor}

\begin{proof}
Define
$S_n(v\circ\pi_\Delta)$ and $S_n(\bar v\circ\bar\pi)$ similarly.

By Propositions~\ref{prop-bar} and~\ref{prop-vv}, $\bar v$ is a mean zero
Lipschitz observable for the nonuniformly
expanding map $\bar f:\bar\Delta\to\bar\Delta$ on the metric space
$(\bar\Delta,d_\theta)$.  Moreover, $\|\bar v\|_\theta\ll \|v\|_\eta$.
Hence, by Corollary~\ref{cor-moment},
$\big|\max_{j\le n}|S_j\bar v|\big|_{L^{p^*}(\bar\mu_\Delta)}\ll \|v\|_\eta\,n^{\max\{\frac12,\frac1p\}}$.

By~\eqref{eq-vv} and Proposition~\ref{prop-psi},
and using that $\pi_\Delta:\Delta\to\Lambda$ and $\bar\pi:\Delta\to\bar\Delta$ 
are measure-preserving,
\begin{align*}
\big|\max_{j\le n}|S_jv|\big|_{L^{p^*}(\mu)} & =
\big|\max_{j\le n}|S_j(v\circ\pi_\Delta)|\big|_{L^{p^*}(\mu_\Delta)}\le
\big|\max_{j\le n}|S_j(\bar v\circ\bar\pi)|\big|_{L^{p^*}(\mu_\Delta)}+2|\psi|_{L^{p^*}(\mu_\Delta)}
\\ & = \big|\max_{j\le n}|S_j\bar v|\big|_{L^{p^*}(\bar\mu_\Delta)}+2|\psi|_{L^{p^*}(\mu_\Delta)} \ll \|v\|_\eta \,n^{\max\{\frac12,\frac1p\}},
\end{align*}
proving part (a).

Next, by Corollary~\ref{cor-Sigma},
$n^{-1}\int_{\bar\Delta}S_n\bar v\,S_n\bar v^T d\bar\mu_\Delta$ converges.
Also,
\begin{align*}
 & \SMALL \big|\int_\Lambda S_nv\,S_nv^T d\mu-\int_{\bar\Delta}S_n\bar v\,S_n\bar v^T d\bar\mu_\Delta\big|
 \\ & \quad  =  \SMALL \big|\int_\Delta \big( S_n(v\circ\pi_\Delta)\,S_n(v\circ\pi_\Delta)^T -S_n(\bar v\circ\bar\pi)\,S_n(\bar v\circ\bar\pi)^T\big)\, d\mu_\Delta\big|
 \\ & \quad \le |S_n(v\circ\pi_\Delta)-S_n(\bar v\circ\bar\pi)|_{L^2(\mu_\Delta)}
 \big(|S_n(v\circ\pi_\Delta)|_{L^2(\mu_\Delta)} + |S_n(\bar v\circ\bar \pi)|_{L^2(\mu_\Delta)}\big)
\\ & \quad  \le 2|\psi|_{L^2(\mu_\Delta)}
 \big(|S_n(v\circ\pi_\Delta)|_{L^2(\mu_\Delta)} + |S_n(\bar v\circ\bar \pi)|_{L^2(\mu_\Delta)}\big)
 \ll |v|_\eta\|v\|_\eta n^{1/2},
\end{align*}
by Proposition~\ref{prop-psi} and the estimates in the proof of part~(a).
Part (b) follows.
\end{proof}

\subsection{Families of nonuniformly hyperbolic transformations}
\label{subsec-NUHlimit}

Suppose that $T_n:\Lambda_n\to\Lambda_n$, $n\ge0$, is a family of nonuniformly hyperbolic transformations with induced maps $F_n=T_n^{\tau_n}:Y_n\to Y_n$.  
Let $p\ge1$.
We say that this is a {\em uniform family of order $p$} if 
\begin{itemize}

\parskip=-2pt
\item[(i)] The constants $D_0,D_1\ge1$, $\gamma\in(0,1)$ 
can be chosen independent of $n\ge0$.
\item[(ii)]   The family $\{\tau_n^p,\, n\ge0\}$ is uniformly integrable.
\end{itemize}

Let $v_n:\Lambda\to\R^d$ be a family of H\"older observables with
$\int_{\Lambda_n}v_n\,d\mu_n=0$.
We suppose that $\sup_{n\ge0}\|v_n\|_\eta<\infty$.

\begin{prop} \label{prop-NUHmoment}  
$\big|\max_{j\le n}|\sum_{k=0}^{j-1} v_n\circ T_n^k|\big|_{L^{p^*}\!(\mu_n)}\le
C \|v_n\|_\eta\, n^{\max\{\frac12,\frac1p\}}$ for all $n\ge0$,
where $p^*=\max\{p,2(p-1)\}$.
\end{prop}

\begin{proof}
This is immediate from Corollary~\ref{cor-UMESigma}(a).
\end{proof}

Proceeding as in Subsection~\ref{sec-NUH}, we construct
metric spaces $(\bar\Delta_n,d_{\theta,n})$ and
families $\psi_n:\Delta_n\to\R^d$, 
$\bar v_n:\bar\Delta_n\to\R^d$ where 
$|\psi_n|_\infty\ll|v_n|_\eta$,
$\|\bar v_n\|_{\theta,n}\ll\|v_n\|_\eta$ such that
\begin{align*} 
v_n\circ\pi_{\Delta_n}=\bar v_n\circ\bar\pi_n+\psi_n-\psi_n\circ f_n.
\end{align*}

By Corollary~\ref{cor-UMESigma}(b), for $p\ge2$ we can define the family of covariance matrices
\begin{align*}
\Sigma_n
& \SMALL =\lim_{k\to\infty}k^{-1}\int_{\Lambda_n}
\big(\sum_{j=0}^{k-1}v_n\circ T_n^j\big)
\big(\sum_{j=0}^{k-1}v_n\circ T_n^j\big)^T\,d\mu_n, \\
&\SMALL  =\lim_{k\to\infty}k^{-1}\int_{\bar\Delta_n} 
\big(\sum_{j=0}^{k-1}\bar v_n\circ \bar f_n^j\big)
\big(\sum_{j=0}^{k-1}\bar v_n\circ \bar f_n^j\big)^T\,d\bar\mu_{\Delta_n}.
\end{align*}

Let
$W_n(t)=n^{-1/2}\sum_{j=0}^{[nt]-1}v_n\circ T_n^j$.
Define $\cW\subset D([0,\infty),\R^d)$ to be the set of weak limits of $\{W_n,\,n\ge0\}$ and let
$\cS\subset\R^{d\times d}$ be the set of limit points of $\{\Sigma_n,\,n\ge0\}$.

\begin{thm} \label{thm-NUH}
Suppose that $p\ge2$.  Then (i) $\{W_n,\,n\ge0\}$ is tight, (ii)
$W\in \cW$ if and only if $W$ is a Brownian motion 
with covariance matrix in $\cS$.

In particular, if $\lim_{n\to\infty}\Sigma_n=\Sigma\in \R^{d\times d}$,
then $W_n\to_wW$ where $W$ is Brownian motion with covariance $\Sigma$.
\end{thm}

\begin{proof}
Define the process
$ \overline{W}_n(t)  = n^{-1/2}\sum_{j=0}^{[nt]-1} \bar v_n \circ \bar f_n^j$ on $(\bar\Delta_n,\bar\mu_{\Delta_n})$.

By Proposition~\ref{prop-bar},
$\bar f_n:\bar\Delta_n\to\bar\Delta_n$
is a uniform family of nonuniformly
expanding maps.
By Proposition~\ref{prop-vv}, $\bar v_n$ is a family of mean zero
Lipschitz observables 
satisfying $\|\bar v_n\|_{\theta,n}\ll\|v_n\|_\eta$.
Hence Theorem~\ref{thm-WIP} characterises the weak limits of 
$\{\overline{W}_n,\,n\ge0\}$.

It remains to show that the weak limits of 
$\{W_n,\,n\ge0\}$ coincide with those of
$\{\overline{W}_n,\,n\ge0\}$.
Since $\pi_{\Delta_n}$ and $\bar\pi_n$ are measure-preserving semiconjugacies,
the weak limits of 
$\{W_n\}$ coincide with those of
$\{W_n\circ\pi_{\Delta_n}\}$, and
the weak limits of 
$\{\overline{W}_n\}$ coincide with those of
$\{\overline{W}_n\circ\bar\pi_n\}$.
Also 
\[
\SMALL \sup_{[0,T]}|W_n\circ \pi_{\Delta_n}-\overline{W}_n\circ\bar\pi_n|_\infty\le 2n^{-1/2}|\psi_n|_\infty\ll n^{-1/2}|v_n|_\eta, 
\]
so
$W_n\circ\pi_{\Delta_n}-\overline{W}_n\circ\bar\pi_n\to_{\mu_{\Delta_n}}0$
completing the proof.
\end{proof}

\begin{examp}  \label{ex-solenoid}
The classical solenoid construction of Smale \& Williams can be used
as in~\cite{AlvesPinheiro08} to construct nonuniformly hyperbolic families
from each of the nonuniformly expanding families in the examples in Section~\ref{sec-NUElimit}.  It is immediate that our results apply to such families.
\end{examp}

\begin{examp} \label{ex-billiard}
Collision maps for dispersing billiards under small external forces are nonuniformly hyperbolic with uniform constants for all $p$ by~\cite{Chernov01}.   
(See the proof of~\cite[Proposition~6.4]{Chernov01} where uniformity of constants is mentioned explicitly.)
Hence the results in this section apply to such examples.
\end{examp}

%
%
%

\section{An abstract homogenization theorem}
\label{sec-homog}

In~\cite{GM13b}, a homogenization theorem was proved for fast-slow systems
of the form 
\[
\x(n+1)=\x(n)+\eps^2 a(\x(n),y(n))+\epsilon b(\x(n))\,v(y(n)),
\]
where the fast dynamics $y(n+1)=Ty(n)$ is generated by an ergodic transformation $T:\Lambda\to\Lambda$
and the slow variables $x(n)$ lie in $\R^d$.
The main assumptions are that
either $d=1$ or $b\equiv I_d$ (or more generally that $b$ is exact, see below), and that $v:\Lambda\to\R^d$ is mean zero and satisfies the WIP for $T$.
The corresponding result for flows was obtained in~\cite{MS11}.

In this section, we show how to generalise to the case where the single map $T$ generating the fast dynamics is replaced by a family of maps.  
As in~\cite{GM13b,MS11}, the setting is completely abstract, with no hyperbolicity assumptions on the fast dynamics.

Let $T_\eps:M \to M$, $\eps\in[0,\eps_0)$, be a family of 
maps defined on a topological space $M$, with $T_\eps$-invariant Borel probability measures $\mu_\eps$.
Consider the family of fast-slow equations
\begin{align} \label{eq-map}
\x(n+1)=\x(n)+\eps^2 a_\eps(\x(n),\y(n))+\epsilon b_\eps(\x(n))\,v_\eps(\y(n)),
\quad \x(0)=\xi_\eps,
\end{align}
where the fast dynamics is given by $\y({n+1})=T_\eps\,\y(n)$,
and $\xi_\eps\in\R^d$ is the initial condition for the slow dynamics.
The maps $v_\eps:M \to\R^d$, $a_\eps: \R^d \times M \to \R^d$ 
and $b_\eps:\R^d\to\R^{d\times d}$ are defined and continuous
for each $\eps\in[0,\eps_0)$.
Moreover, $\int_M v_\eps\,d\mu_\eps=0$.

\vspace{-1ex}
\paragraph{Regularity assumptions}
We suppose that there is a constant $L\ge1$ such that
\begin{align*}
& |v_\eps|_\infty\le L, \quad
 |a_\eps|_\infty\le L, \quad 
\Lip\,a_\eps=\sup_{x\neq x'}\sup_y\frac{|a_\eps(x,y)-a_\eps(x',y)|}{|x-x'|}\le L,
\end{align*}
for all $\eps\in[0,\eps_0)$.
Also, we assume that
\begin{align*}
& \lim_{\eps\to0}\sup_{x,y}|a_\eps(x,y)-a_0(x,y)|=0,
\qquad
 \lim_{\eps\to0}\sup_{y}|v_\eps(y)-v_0(y)|=0,
\qquad 
\lim_{\eps\to0}\xi_\eps=\xi_0.
\end{align*}

\vspace{-1ex}
\paragraph{Exactness}
We suppose that the multiplicative noise
$b_\eps:\R^d\to \R^{d\times d}$ 
is {\em exact}.  That is, $b_\eps(x)=[(dh_\eps)_x]^{-1}$
where $h_\eps:\R^d\to\R^d$ is a continuous
family of $C^3$ diffeomorphisms $h_\eps:\R^d\to\R^d$ (with $C^3$ norm uniform in
$\eps$ and $x\in\R^d$).

\begin{rmk} \label{rmk-setup}
(a) The exactness assumption on $b_\eps$ can be removed, but then additional assumptions are required on the fast dynamics.  When the fast dynamics $T_\eps$ is independent of $\eps$,
the corresponding result without exactness is proved 
for partially hyperbolic dynamics using standard pairs/martingale problems~\cite{Dolgopyat04}, and for
nonuniformly hyperbolic dynamics using rough path theory~\cite{KM16,KMsub}.

\noindent(b)  The global nature of the regularity assumptions in $x\in\R^d$ is easily relaxed, see for example~\cite[Section~3.1]{GM13b}.
\end{rmk}

Define $\hatx(t)=\x([t\eps^{-2}])$, $t\ge0$.
We are interested in weak convergence of $\hatx$ to a solution $X$ of an SDE.
Convergence is  in the space
$D([0,\infty),\R^d)$ of c\`adl\`ag functions (right-continuous functions with left-hand limits, see for example~\cite[Chapter~3]{Billingsley99})
with the supremum norm.

\begin{rmk}  \label{rmk-sep}
One technical issue is that $D([0,\infty),\R^d)$ is not separable in the supremum norm.  Since our limit processes have continuous sample paths, convergence in the supremum norm is equivalent to convergence in the standard Skorokhod topology which is metrizable and separable.  Hence whenever we apply results where separability is required, we can momentarily work in this topology.
\end{rmk}

\vspace{-1ex}
\paragraph{Dynamical assumptions}
So far, the assumptions on the fast-slow system~\eqref{eq-map} have been standard subject to the comments in
Remark~\ref{rmk-setup}.  Now we introduce mild assumptions on the fast dynamics that suffice for homogenization.

Define $\alpha_x:M\to\R^d$ for $x\in\R^d$,
\begin{align} \label{eq-alphax} 
\SMALL \alpha_x(y)=a_0(x,y)-\frac12\{(db_0)_u b_0(x)v_0(y)\}v_0(y).
\end{align}
Also, define the family of random elements
$W_\eps:(M,\mu_\eps)\to D([0,\infty),\R^d)$,
        \[
W_\eps(t)=\eps\sum_{j=0}^{[t/\eps^2]-1}v_\eps\circ T_\eps^j.
\]

\begin{description}

      \parskip=-3pt

\item[Uniform mean ergodicity (UME)]  
There exists $P:\R^d\to\R^d$ such that
\[
\lim_{\eps\to0}\int_M\Big|\eps^{1/2}\sum_{j=0}^{[\eps^{-1/2}]-1} \alpha_x\circ T_\eps^j\;-P(x)\Big|\,d\mu_\eps=0 \quad\text{for all $x\in\R^d$}.
\]
\item[Weak invariance principle (WIP)]
$W_\eps\to_{\mu_\eps} W$ in $D([0,\infty),\R^d)$ as $\eps\to0$,
where $W$ is Brownian motion with some covariance matrix $\Sigma\in\R^{d\times d}$.
\end{description}

\begin{thm} \label{thm-map}
Assume (UME) and (WIP).  Then $P$ is Lipschitz and
$\hatx\to_{\mu_\eps} X$ in $D([0,\infty),\R^d)$ as \mbox{$\epsilon\to0$} where $X$ is the solution to
the Stratonovich SDE
\begin{align} \label{eq-SDE}
dX=
P(X)\, dt+ b_0(X)\circ dW,
\quad X(0)=\xi.
\end{align}
\end{thm}

\begin{rmk} \label{rmk-UME}
Suppose that 
\begin{itemize}

\parskip = -2pt
\item[(a)] $\mu_\eps\to_w\mu_0$ as $\eps\to0$
(statistical stability).
\item[(b)] $\BIG\lim_{\eps\to0}\int_M\Big|\eps^{1/2}\sum_{j=0}^{[\eps^{-1/2}]-1} \alpha_x\circ T_\eps^j\;-\;\int_M \alpha_x\,d\mu_\eps\Big|\,d\mu_\eps=0$
for all $x\in\R^d$.
\end{itemize}

Then (UME) holds, and
\begin{align*}
P(x) & \SMALL  =\int_M a_0(x,y)\,d\mu_0(y)-\frac12 \int_M\{(db_0)_{x}b_0(x)v_0(y)\}v_0(y)\,d\mu_0(y)
\\ & \SMALL  =\int_M a_0(x,y)\,d\mu_0(y)-\frac12\sum  b_0^{\alpha\gamma}(x)(\partial_{x_\alpha} b_0^\beta)(x)\!\int_M v_0^\beta v_0^\gamma\,d\mu_0.
\end{align*}
Here, $b_0^{\alpha\gamma}$ denotes the $(\alpha,\gamma)$'th entry of $b_0$
and $b_0^\beta$ denotes the $\beta$'th column of $b_0$,
while the summation is over indices $\alpha,\beta,\gamma=1,\dots,d$.
\end{rmk}

\begin{rmk} \label{rmk-map}
(a) 
We focus attention on weak convergence with respect to the family of invariant measures $\mu_\eps$.  If we assume {\em strong statistical stability} (so $\mu_\eps$ is absolutely continuous with respect to a reference measure $\rho$ for all $\eps$ and $d\mu_\eps/d\rho\to d\mu_0/d\rho$ in $L^1$), then it is immediate from Theorem~\ref{thm-map} that $x_\eps\to_{\mu_0}X$.  We will return to the issue of weak convergence with respect to a wider range of measures in subsequent work.
 
\vspace{1ex}
\noindent(b) 
The WIP is a necessary condition for Theorem~\ref{thm-map} since it is equivalent to the case $a_\eps\equiv0$, $b_\eps\equiv I_d$.
\end{rmk}

The remainder of this section is concerned with the proof of Theorem~\ref{thm-map}.  We 
deal first with the special case $b_\eps\equiv I_d$,  and then with the general case.

\vspace{-1ex}
\paragraph{The special case $b_\eps\equiv I_d$}

We extend the arguments in~\cite{GM13b,MS11} developed for the situation where $T_\eps$ is independent of~$\eps$.

\begin{lemma} \label{lem-map}
Theorem~\ref{thm-map} holds 
in the case $b_\eps\equiv I_d$.
\end{lemma}

\begin{proof}
Note that $\alpha_x(y)=a_0(x,y)$ is uniformly Lipschitz in $x$ and hence
that $P$ is Lipschitz with constant $L$.
Define $\tilde a(x,y)=a_0(x,y)-P(x)$.

To prove weak convergence in $D([0,\infty),\R^d)$, it suffices to prove convergence in $D([0,T],\R^d)$ for each fixed $T\ge1$.
Let $\delta(\epsilon)=\sup_{x,y}|a_\eps(x,y)-a_0(x,y)|+|\xi_\eps-\xi|$, so
$\lim_{\epsilon\to0}\delta(\epsilon)=0$.
Write
\[
\x(n)=\xi_\eps+ \epsilon^2\sum_{j=0}^{n-1}a_\eps(\x(j),\y(j))
+ \epsilon\sum_{j=0}^{n-1}v_\eps(\y(j)).
\]
Hence
\begin{align*}
\hatx(t)& =\xi+
\epsilon^2\sum_{j=0}^{[t\epsilon^{-2}]-1}a_0(\x(j),\y(j))
+W_\eps(t)+Z_{\eps,1}(t) \\ & =
\xi +\epsilon^2\sum_{j=0}^{[t\epsilon^{-2}]-1}P(\hatx(\epsilon^2j))
+W_\eps(t)+ Z_{\eps,1}(t)+
Z_{\eps,2}(t),
\end{align*}
where
$W_\eps(t)=\epsilon\sum_{j=0}^{[t\epsilon^{-2}]-1}v_\eps(\y(j))$ satisfies the WIP and
\[
|Z_{\eps,1}(t)|\le T\delta(\epsilon), \quad
Z_{\eps,2}(t)=
\epsilon^2\sum_{j=0}^{[t\epsilon^{-2}]-1} \tilde a(\x(j),\y(j)).
\]
For $t$ an integer multiple of $\epsilon^2$,
the term $\epsilon^2\sum_{j=0}^{[t\epsilon^{-2}]-1}P(\hatx(\epsilon^2j))$ is the Riemann sum of a piecewise constant function and is precisely
$\int_0^t P(\hatx(s))\,ds$.
For general $t$,
\[
\epsilon^2\sum_{j=0}^{[t\epsilon^{-2}]-1}P(\hatx(\epsilon^2j))=
\int_0^t P(\hatx(s))\,ds+Z_{\eps,3}(t),
\]
where $|Z_{\eps,3}(t)|\le \epsilon^2|P|_\infty\le \eps^2L$.
Altogether,
\[
\SMALL \hatx(t)=\xi+
\int_0^t P(\hatx(s))\,ds+
W_\eps(t)+Z_\eps(t), 
\]
where $Z_\eps=Z_{\eps,1}+ Z_{\eps,2}+ Z_{\eps,3}$.

We show below that
$Z_{\eps,2}\to_{\mu_\eps}0$ in $D([0,T],\R^d)$.
It follows that $W_\eps+ Z_\eps\to_{\mu_\eps} W$ in $D([0,T],\R^d)$.

Now consider the continuous map
$\mathcal{G}:D([0,T],\R^d)\to D([0,T],\R^d)$ given by $\mathcal{G}(u)=z$
where $z$ is the unique solution to the integral equation
\[
\SMALL z(t)=\xi+u(t)+\int_0^t P(z(s))\,ds.  
\]
Define $\z= \mathcal{G}(W_\eps+Z_\eps)$.
Since continuous maps preserve weak convergence, it follows that 
$\z\to_{\mu_\eps} \mathcal{G}(W)$.   But $z_\eps=\x$ by uniqueness of solutions,
so $\x\to_{\mu_\eps} \mathcal{G}(W)$.
The result follows since $X=\mathcal{G}(W)$ satisfies the SDE
$dX=P(X)\,dt+dW$, $X(0)=\xi$.

It remains to show that
$Z_{\eps,2}\to_{\mu_\eps}0$ in $D([0,T],\R^d)$.
Note that $|\tilde a|_\infty\le 2L$ and $\Lip\,\tilde a\le 2L$.   
Let $N=[t\eps^{-3/2}]$ and write $Z_{\eps,2}=Y_\eps+I_0$, where 
\begin{align*}
Y_\eps(t) & =\eps^2\!\!\!\!\sum_{0\le j< N\eps^{-1/2}} \tilde a(\x(j),\y(j)), \qquad
I_0(t) =\eps^2\!\!\!\!\!\!\!\!\sum_{N\eps^{-1/2}\le j\le [t\eps^{-2}]-1} \tilde a(\x(j),\y(j)).  
\end{align*}
We have 
\begin{align} \label{eq-I0}
|I_0(t)|\le \eps^{3/2}|\tilde a|_\infty \le 2L\eps^{3/2}.
\end{align}

We now estimate $Y_\eps$ as follows:
\begin{align*}
& Y_\eps(t) = \eps^2\sum_{n=0}^{N-1} \sum_{n\eps^{-1/2}\le j<(n+1)\eps^{-1/2}} 
\tilde a(\x(j),\y(j))=I_1+I_2 \\
&  I_1=\eps^2\sum_{n=0}^{N-1}
\sum_{n\eps^{-1/2}\le j<(n+1)\eps^{-1/2}}\big(\tilde a(\x(j),\y(j))-\tilde a(\x(n\eps^{-1/2}),\y(j))\big) \\ &
 I_2 =\eps^2\sum_{n=0}^{N-1}
\sum_{n\eps^{-1/2}\le j<(n+1)\eps^{-1/2}} \tilde a(\x(n\eps^{-1/2}),\y(j)).
\end{align*}

For $n\eps^{-1/2}\le j<(n+1)\eps^{-1/2}$, we have
$|\x(j)-\x(n\eps^{-1/2})| \le 
(|a_\eps|_{\infty}+|v_\eps|_{\infty})\eps^{1/2} \le 2L\eps^{1/2}$.
Hence 
\begin{align} \label{eq-I1}
|I_1|\le N\eps^{3/2} \Lip\,\tilde a\,2L\eps^{1/2}\le 
4L^2T\eps^{1/2}.
\end{align}

Next, $I_2  =\eps^{3/2}\sum_{n=0}^{N-1}J_n$, where
\begin{align*}
J_n  &= \eps^{1/2}\sum_{n\eps^{-1/2}\le j<(n+1)\eps^{-1/2}} \tilde a(\x(n\eps^{-1/2}),\y(j)).
\end{align*}
Hence
\begin{align} \label{eq-J}
|I_2|\le \eps^{3/2}\sum_{n=0}^{[T\eps^{-3/2}]-1}|J_n|.
\end{align}

For $u\in\R^d$ fixed, define
\begin{align*}
	\tildeJ_n(u)  =
 \eps^{1/2}\sum_{n\eps^{-1/2}\le j<(n+1)\eps^{-1/2}} \tilde a(u,\y(j))
& = \eps^{1/2}\sum_{n\eps^{-1/2}\le j<(n+1)\eps^{-1/2}} \alpha_u\circ T_\eps^j
\;-\;P(u).
\end{align*}
Note that $\tildeJ_0$ has $[\eps^{-1/2}]$ terms, and
$\tildeJ_n(u)$ has at most one term more or one term less than $\tildeJ_0(u)$.
Hence
\[
\SMALL \int_M|\tildeJ_n(u)|\,d\mu_\eps =\int_M|\tildeJ_0(u)|\,d\mu_\eps+E_n(u),
\quad\text{where}\quad |E_n(u)|\le \eps^{1/2}|\tilde a|_\infty \le 2L\eps^{1/2}.
\]

Let $Q>0$ and write $I_2=K_{Q,1}+K_{Q,2}$ where
\begin{align*}
& K_{Q,1}=I_21_{B_\eps(Q)}, \quad
K_{Q,2}=I_21_{B_\eps(Q)^c}, \quad
 B_\eps(Q)=\big\{\max_{[0,T]}|\x|\le Q\big\}.
\end{align*}

For any $\sigma>0$, there exists a finite subset $S\subset\R^d$ such that
$\dist(x,S)\le \sigma/(2L)$ for any $x$ with $|x|\le Q$.
Then 
$1_{B_\eps(Q)}|J_n|\le \sum_{u\in S}|\tildeJ_n(u)|+\sigma$
for all $n\ge0$, $\eps>0$,
Hence by~\eqref{eq-J},
\begin{align*}
 \int_M\max_{[0,T]}|K_{Q,1}|\,  d\mu_\eps &
 \le \eps^{3/2}\sum_{n=0}^{[T\eps^{-3/2}]-1}
\sum_{u\in S}\int_M|\tildeJ_n(u)|\,d\mu_\eps+T\sigma
\\ & = \eps^{3/2}\sum_{n=0}^{[T\eps^{-3/2}]-1}
\sum_{u\in S}\Big(\int_M|\tildeJ_0(u)|\,d\mu_\eps+E_n(u)\Big)+T\sigma
  \\ & \le T \sum_{u\in S}\int_M|\tildeJ_0(u)|\,d\mu_\eps
+2\eps^{1/2}T|S|L+T\sigma.
\end{align*}
By (UME), $\int_M|\tildeJ_0(u)|\,d\mu_\eps\to0$ as $\eps\to0$ for each $u$.
Since $\sigma>0$ is arbitrary, we obtain for each fixed $Q$ that
$\max_{[0,T]}|K_{Q,1}|\to0$ in $L^1(\mu_\eps)$, and hence in probability, as $\eps\to0$.

Next, since $\x-W_\eps$ is bounded on $[0,T]$, for $Q$ sufficiently large
\[
\mu_\eps\big\{\max_{[0,T]}|K_{Q,2}|>0\big\}\le 
\mu_\eps\big\{\max_{[0,T]}|\x|\ge Q\big\}\le 
\mu_\eps\big\{\max_{[0,T]}|W_\eps|\ge Q/2\big\}.
\]
Fix $c>0$.  Increasing $Q$ if necessary, we can arrange that
$\mu_\eps\{\max_{[0,T]}|W|\ge Q/2\}<c/4$.
By the continuous mapping theorem,
$\max_{[0,T]}|W_\eps|\to_d \max_{[0,T]}|W|$.
Hence there exists $\eps_1>0$ such that
$\mu_\eps\{\max_{[0,T]}|W_\eps|\ge Q/2\}<c/2$ for all $\eps\in(0,\eps_1)$.
For such $\eps$,
\[
\mu_\eps\big\{\max_{[0,T]}|K_{Q,2}|>0\big\}<c/2.
\]
Shrinking $\eps_1$ if necessary, we also have that
$\mu_\eps\{\max_{[0,T]}|K_{Q,1}|>c/2\}<c/2$.
Hence $\mu_\eps\{\max_{[0,T]}|I_2|>c\}<c$, and so 
$\max_{[0,T]}|I_2|\to0$ in probability.  
Combining this with estimates~\eqref{eq-I0} and~\eqref{eq-I1}, we obtain that 
$\max_{[0,T]}|Z_{\eps,2}|\to0$ in probability as required.~
\end{proof}

\vspace{-1ex}
\paragraph{The case of exact noise}
Now we consider the general case of exact multiplicative noise, following~\cite{GM13b}.

\begin{pfof}{Theorem~\ref{thm-map}}
Define $\z(n)= h_\eps(\x(n))$.
Using Taylor's theorem to expand the $C^3$ map $h_\eps$, we obtain
\begin{align} \label{eq-r}
\nonumber    & \z(n+1)  -\z(n)    =  h_\eps(x^{(\epsilon)}(n+1))-h_\eps(x^{(\epsilon)}(n)) \\[.75ex]
 & = (dh_\eps)_{\x(n)}\big(\x(n+1)-\x(n)\big) 
\\[.75ex]\nonumber  & 
+ {\SMALL\frac12} \{(d^2h_\eps)_{\x(n)}
(\x(n+1)-\x(n))\} (\x(n+1)-\x(n))
 +o(|\x(n+1)-\x(n)|^2). 
\end{align}
Here, we are identifying $(d^2h)_x$ as an element of
$L(\R^d,L(\R^d,\R^d))$ for each $x\in\R^d$.
The last term is uniformly $o(\epsilon^2)$.

Substituting for $\x(n+1)-\x(n)$ using equation~\eqref{eq-map} and the fact that
$b_\eps=[dh_\eps]^{-1}$, equation~\eqref{eq-r} becomes
\begin{align*}
 \z(n+1)- \z(n)  & = \epsilon^2 
\Big\{(dh_\eps)_{\x(n)}a_\eps(\x(n),\y(n))  \\ &
\!\!\!\!  \!\!\!\!  \!\!\!\!  \!\!\!\!
\!\!\!\!  \!\!\!\!  \!\!\!\!  \!\!\!\!
+{\SMALL\frac12} \{(d^2h_\eps)_{\x(n)}
b_\eps(\x(n))v_\eps(\y(n))\}
b_\eps(\x(n))v_\eps(\y(n))
+o(1)\Big\}
+\eps v_\eps(\y(n)).
\end{align*}
In other words,
\[
 \z(n+1)- \z(n)  = \epsilon^2 A_\eps(\z(n),\y(n))+ \eps v_\eps(\y(n)), 
\]
where
\[
A_\eps(z,y)=(dh_\eps)_{h_\eps^{-1}\!z}\Big\{a_\eps(h_\eps^{-1}z,y)-\frac12
\{(db_\eps)_{h_\eps^{-1}\!z}b_\eps(h_\eps^{-1}\!z)v_\eps(y)\}
v_\eps(y)
+o(1)\Big\}
\]
uniformly in $z,y$ as $\epsilon\to0$.

The regularity assumptions on $v_\eps$, $a_\eps$ and $h_\eps$ ensure that 
$A_\eps$ is bounded and globally Lipschitz in $z$.
Similarly, it is easily checked that
$\lim_{\eps\to0}\sup_{z,y}|A_\eps(z,y)-A_0(z,y)|=0$.
Notice also that $A_0(z,y)=(dh_0)_{h_0^{-1}z}\alpha_{h_0^{-1}z}(y)$.

Hence we are in the situation of Lemma~\ref{lem-map}, and it follows
that $\tildez(t)=\z([t\epsilon^{-2}])$
converges weakly to solutions $Z$ of the SDE
\begin{align} \label{eq-SDE_Z_map}
dZ=(dh_0)_{h_0^{-1}z}P(h_0^{-1}Z)\,dt+dW,
\end{align}
where $P(Z)$ is the limit function in (UME).

Next
\[
\SMALL\sup_t|h_\eps^{-1}(\z(t))-h_0^{-1}(\z(t))|
\le \sup_z|h_\eps^{-1}(z)-h_0^{-1}(z)|\to0,
\]
so by the continuous mapping theorem, 
\[
\x=h_\eps^{-1}(\z)=h_0^{-1}(\z)+\{h_\eps^{-1}(\z)-h_0^{-1}(\z)\}\to_{\mu_\eps} h_0^{-1}(Z).
\]
Hence it remains to determine $X=h_0^{-1}(Z)$.
Clearly $X(0)=\xi$.
Since the Stratonovich integral transforms according to the standard laws of calculus, 
\begin{align*}
dX & =[(dh_0)_X]^{-1}\circ dZ
 = [(dh_0)_X]^{-1}\circ [(dh_0)_{h_0^{-1}Z}P(h_0^{-1}Z)\,dt+\,dW]
\\ & =P(X)\,dt+b_0(X)\circ dW,
\end{align*}
as required.
\end{pfof}

\section{Homogenization for uniform families of fast-slow systems}
\label{sec-homog-uniform}

In this section, we apply the abstract homogenization theorem from 
Section~\ref{sec-homog} to the case where the fast dynamics is generated by
a uniform family of nonuniformly hyperbolic transformations.
We consider first the noninvertible case (Subsection~\ref{sec-homog-NUE})
and then the invertible case (Subsection~\ref{sec-homog-NUH}).

\subsection{Nonuniformly expanding fast dynamics}
\label{sec-homog-NUE}

Let $(M,d_M)$ be a bounded metric space with finite Borel measure $\rho$.
For each $\eps\in[0,\eps_0)$ we suppose that $T_\eps:M\to M$ is a nonuniformly
expanding map as in Section~\ref{subsec-NUE},
with induced map $F_\eps=T_\eps^{\tau_\eps}:Y_\eps\to Y_\eps$,  and absolutely continuous
ergodic $T_\eps$-invariant and $F_\eps$-invariant Borel probability measures $\mu_\eps$ and $\mu_{Y_\eps}$.
(The 
metric space $(M,d_M)$ and finite Borel measure $\rho$ are fixed independent of $\eps$.  This is natural for the purposes of this current section, but is easily relaxed, see Remark~\ref{rmk-Sigma}.)

We assume that $T_\eps$ is 
a uniform family of order $p\ge2$ (cf.\ Section~\ref{sec-NUElimit}),
so the  various constants in the definition of nonuniformly expanding can be chosen
independent of $\eps\in[0,\eps_0)$, and $\{\tau_\eps^2,\,\eps\in[0,\eps_0)\}$
is uniformly integrable.
Moreover, we suppose that $\mu_0$ is {\em statistically stable}:
$\mu_\eps\to_w\mu_0$ as $\eps\to0$.

Let $v_\eps:M\to\R^d$, $\eps\in[0,\eps_0)$, be a family of H\"older observables with
$\int_M v_\eps\,d\mu_\eps=0$.
We require that $v_\eps$ and $T_\eps$ satisfy
\begin{align} \label{eq-v}
& \SMALL \sup_{\eps\in[0.\eps_0)}\|v_\eps\|_\eta<\infty, \qquad
\lim_{\eps\to0}|v_\eps-v_0|_\infty=0,
\end{align}
and
\begin{align}
\label{eq-tech}
& \SMALL \int_M v_0\circ T_0^j\,(v_0\circ T_0^k)^T\,(d\mu_\eps-d\mu_0)\to 0, \qquad
 T_\eps^j\to_{\mu_\eps}T_0^j, 
\end{align} 
for all $j,k\ge0$, as $\eps\to0$.
(The last part of condition~\eqref{eq-tech} means that
$\mu_\eps\{y\in M:d_M(T_\eps^jy,T_0^jy)>a\}\to0$ for all $a>0$.)

Consider the family of fast-slow equations~\eqref{eq-map} where
$\y(n+1)=T_\eps\y(n)$.
We assume that $a_\eps$ and $b_\eps$ satisfy the regularity conditions
in
Section~\ref{sec-homog} and that $b_\eps$ is exact.
Let $\hat x_\eps=x_\eps([t\eps^{-2]}])$.

\begin{thm} \label{thm-NUE}
Let $P(x)=
\int_M a_0(x,y)\,d\mu_0(y)-\frac12 \int_M\{(db_0)_{x}b_0(x)v_0(y)\}v_0(y)\,d\mu_0(y)$.
Let $W$ denote $d$-dimensional Brownian motion with covariance
\[ \SMALL
\Sigma=
\lim_{n\to\infty}\frac1n\int_M 
\big(\sum_{j=0}^{n-1}v_0\circ T_0^j\big)
\big(\sum_{j=0}^{n-1}v_0\circ T_0^j\big)^T\,d\mu_0.
\]

Then 
$\hatx\to_{\mu_\eps} X$ in $D([0,\infty),\R^d)$ as \mbox{$\epsilon\to0$} where $X$ is the solution to
the Stratonovich SDE
\begin{align*} 
dX=
P(X)\, dt+ b_0(X)\circ dW,
\quad X(0)=\xi.
\end{align*}
\end{thm}

\begin{rmk} \label{rmk-sss}  
In very general situations,~\cite{Alves04,AlvesViana02} show that $\mu_0$ is strongly statistically stable.  The first part of condition~\eqref{eq-tech} follows immediately.
Moreover, in the conclusion of Theorem~\ref{thm-NUE} we obtain in addition that
$\hat x_\eps\to_{\mu_0}X$ by Remark~\ref{rmk-map}(a).
\end{rmk}

\begin{examps}  It is straightforward to choose the examples in Section~\ref{sec-NUElimit} to be strongly statistically stable.  Theorem~\ref{thm-NUE}
and Remark~\ref{rmk-sss} then apply.

For instance, in the case of the intermittent maps, Example~\ref{ex-LSV2},
fix $\gamma_0\in (0,\frac12)$ and choose $\gamma_\eps\to\gamma_0$.
Let $T_\eps$, $0\le \eps<\eps_0$, be the corresponding family of intermittent maps.  Then $\mu_0$ is strongly statistically stable by~\cite{BaladiTodd16,Korepanov16},
while (UME) and (WIP) follow from Section~\ref{sec-NUElimit}.

Similar comments apply to Examples~\ref{ex-CE} and~\ref{ex-Viana}
with statistical stability following from~\cite{FreitasTodd09} and~\cite{AlvesViana02} respectively (see the corresponding examples in~\cite{KKMsub} for details).
\end{examps}

\begin{rmk} \label{rmk-Sigma}
Various conditions --- namely independence of $M$ and $\rho$ on $\eps$, $\lim_{\eps\to0}|v_\eps-v_0|_\infty=0$, and conditions~\eqref{eq-tech} --- are used only in the 
proof of continuity of certain covariance matrices $\Sigma_\eps$, see Proposition~\ref{prop-Sigma}.
It is easy to check that the results in this section go through 
with these assumptions removed, provided $\diam M$ is bounded independent of $\eps$ and the conclusion of Proposition~\ref{prop-Sigma} holds.
We note that~\cite{DemersZhang13} gives general conditions under which $\Sigma_\eps$ varies continuously.
\end{rmk}

In the remainder of this subsection, we prove Theorem~\ref{thm-NUE}.
By Theorem~\ref{thm-map}, it suffices to verify (UME) and (WIP) and to identify $P$ and $\Sigma$.

\begin{prop} \label{prop-UME}
Condition (UME) is satisfied
with $P(x)$ as stated in Theorem~\ref{thm-NUE}.
\end{prop}

\begin{proof}
We apply Remark~\ref{rmk-UME}.
Condition~(a) is automatic, so it remains to verify condition~(b).

Recall from~\eqref{eq-alphax} that
$\alpha_x(y)=a_0(x,y)-\frac12\{(db_0)_xb_0(x)v_0(y)\}v_0(y)$.
Define $\beta_{x,\eps}=\alpha_x-\int_M\alpha_x\,d\mu_\eps$.
Then $\beta_{x,\eps}:M\to\R^d$ is family of H\"older observables with
$\int_M \beta_{x,\eps}\,d\mu_\eps=0$ such that
$\|\beta_{x,\eps}\|_\eta\ll \|v_0\|_\eta^2$ uniformly in $\eps$.
(The estimate is also uniform in $x$, but that is not needed.)
By Lemma~\ref{lem-moment},
$\lim_{\eps\to0}\int_M|\sum_{j=0}^{[\eps^{-1/2}]-1} \beta_{x,\eps}\circ T_\eps^j|\,d\mu_\eps=0$ for all $x\in\R^d$
as required.~
\end{proof}

As in Section~\ref{sec-NUElimit}, we can define
the family of covariance matrices 
\begin{align*}  
\SMALL \Sigma_\eps
 =\lim_{n\to\infty}n^{-1}\int_M S_nv_\eps\, S_nv_\eps^T\,d\mu_\eps, \quad
S_nv_\eps = \sum_{j=0}^{n-1}v_\eps\circ T_\eps^j, \quad\eps\in[0,\eps_0).
\end{align*}

\begin{prop} \label{prop-Sigma}
$\lim_{\eps\to0}\Sigma_\eps=\Sigma_0$.
\end{prop}

\begin{proof}
  Write $I_{\eps,n}=\int_M S_nv_\eps\, S_nv_\eps^T \, d\mu_\eps$.

Let $\delta>0$.  
  By Remark~\ref{rmk-unifSigma}, there exists $N\ge1$ such that
      $|N^{-1}I_{\eps,N}-\Sigma_\eps |<\delta$ for all $\eps\in[0,\eps_0)$.
Next
  \[ \SMALL
    I_{\eps,N}-I_{0,N}
     = \int_M ( S_Nv_\eps\, S_Nv_\eps^T - S_Nv_0\, S_Nv_0^T) \, d \mu_\eps
    + \int_M S_Nv_0\, S_Nv_0^T \, (d\mu_\eps - d\mu_0).
  \]
  By condition~\eqref{eq-tech},
  \(\lim_{\eps \to 0} \int_M S_Nv_0\, S_Nv_0^T \, (d\mu_\eps - d\mu_0) = 0\).
Also,
\begin{align*}
| S_Nv_\eps\, S_Nv_\eps^T - S_Nv_0\, S_Nv_0^T|_{L^1(\mu_\eps)} & 
\le \big(|S_Nv_\eps|_{L^2(\mu_\eps)}+|S_Nv_0|_{L^2(\mu_\eps)}\big)
|S_Nv_\eps-S_Nv_0|_{L^2(\mu_\eps)}
\\ & \le N(|v_\eps|_\infty+|v_0|_\infty)|S_Nv_\eps-S_Nv_0|_{L^2(\mu_\eps)}.
\end{align*}
But
\[ \SMALL
|S_Nv_\eps-S_Nv_0|  \le \sum_{j=0}^{N-1}|v_\eps-v_0|\circ T_\eps^j+
\sum_{j=0}^{N-1}|v_0\circ T_\eps^j-v_0\circ T_0^j|
 \le N|v_\eps-v_0|_\infty +|v_0|_\eta g_{\eps,N},
\]
where $g_{\eps,N}(y)=\sum_{j=0}^{N-1}d_M(T_\eps^jy,T_0^jy)^\eta$.
By~\eqref{eq-v} and condition~\eqref{eq-tech}, we obtain that
$\lim_{\eps\to0}| S_Nv_\eps\, S_Nv_\eps^T - S_Nv_0\, S_Nv_0^T|_{L^1(\mu_\eps)}=0$.
Hence $\lim_{\eps\to0}I_{\eps,N}=I_{0,N}$ and so
$\limsup_{\eps\to0}|\Sigma_\eps-\Sigma_0|<2\delta$.
  Since $\delta$  is arbitrary, the result follows.
\end{proof}

\begin{cor}   Condition (WIP) holds with $\Sigma$
as stated in Theorem~\ref{thm-NUE}.
\end{cor}

\begin{proof}
This follows immediately from Propositions~\ref{prop-limit} and~\ref{prop-Sigma}.
\end{proof}

\subsection{Nonuniformly hyperbolic fast dynamics}
\label{sec-homog-NUH}

Now we show how to extend Theorem~\ref{thm-NUE} to the invertible setting.
We assume the same set up as in Subsection~\ref{sec-homog-NUE} except that
$T_\eps$ is now a uniform family of nonuniformly hyperbolic transformations 
as in Subsection~\ref{sec-NUH} 

\begin{thm} \label{thm-NUHhomog}
The conclusion $\hat x_\eps\to_{\mu_\eps}X$ of Theorem~\ref{thm-map} remains valid.
If in addition $\mu_0$ is strongly statistically stable then
$\hat x_\eps\to_{\mu_0}X$.
\end{thm}

\begin{rmk} \label{rmk-Sigma2}  
The comments in Remark~\ref{rmk-Sigma} apply equally in the current context.
\end{rmk}

To prove Theorem~\ref{thm-NUHhomog}, it again suffices to
verify condition (UME) and (WIP)  in Theorem~\ref{thm-map}.

The proof of
(UME) is identical to that of Proposition~\ref{prop-UME} with Corollary~\ref{cor-UMESigma}(a) replacing Lemma~\ref{lem-moment}.
By Corollary~\ref{cor-UMESigma}(b), we can define
\begin{align*}
\Sigma_\eps
& \SMALL =\lim_{n\to\infty}\frac1n\int_M 
\big(\sum_{j=0}^{n-1}v_\eps\circ T_\eps^j\big)
\big(\sum_{j=0}^{n-1}v_\eps\circ T_\eps^j\big)^T\,d\mu_\eps.
\end{align*}
By Remark~\ref{rmk-unifSigma} and the proof of Corollary~\ref{cor-UMESigma}(b), 
the convergence is uniform in~$\eps$.  Hence 
the argument in the proof of Proposition~\ref{prop-Sigma} shows that
$\lim_{\eps\to0}\Sigma_\eps=\Sigma_0$.
Condition (WIP) with $\Sigma=\Sigma_0$ follows from Theorem~\ref{thm-NUH}.

\begin{examp}
By~\cite{DemersZhang13}, statistical stability holds for the families of externally forced dispersing billiards in Example~\ref{ex-billiard},
and hence Theorem~\ref{thm-NUHhomog} holds.  

We note that the stronger linear response property can be established in certain situations~\cite{ChernovKorepanov14}, but that linear response is not required for the purposes of this paper.
\end{examp}

\appendix

\section{WIP for martingale difference arrays}
\label{sec-mda}

In this appendix, we recast a classical martingale CLT/WIP of~\cite{Billingsley99} into a form that is convenient  for ergodic stationary martingale difference arrays of the type commonly encountered in the deterministic setting.  

Let $\{(\Delta_n,\cM_n,\mu_n)\}$ be a sequence of probability spaces.
Suppose that $f_n:\Delta_n\to\Delta_n$ is a sequence of measure-preserving transformations with transfer operators $L_n$
and Koopman operators $U_n$.
Suppose that $m_n:\Delta_n\to\R^d$ lies in $L^2(\Delta_n)$ and that
$\int_{\Delta_n}m_n\,d\mu_n=0$ and $m_n\in\ker L_n$.

Define the sequence of processes $M_n:\Delta_n\to D([0,\infty),\R^d)$
by
\[
M_n(t)=n^{-1/2}\sum_{j=0}^{[nt]-1}m_n\circ f_n^j, \quad t\ge0.
\]

\begin{thm} \label{thm-mda}
Suppose that the family $\{|m_n|^2,\,n\ge0\}$ is uniformly integrable.
Suppose also that
there exists a constant matrix 
$\Sigma\in\R^{d\times d}$, such that
$n^{-1}\sum_{j=0}^{[nt]-1}\{U_nL_n(m_nm_n^T)\}\circ f_n^j\to_{\mu_n} t\Sigma$
as $n\to\infty$ for each $t>0$.

Then $M_n\to_{\mu_n} W$ in $D([0,\infty),\R^d)$ where $W$ is Brownian motion with covariance~$\Sigma$.
\end{thm}

\begin{proof}
First we consider the case where $\Sigma$ is not identically zero.
By Prokhorov's Theorem, we must prove convergence of finite-dimensional distributions and tightness.

\noindent{\bf Finite-dimensional distributions}
Fix $0= t_0\le t_1\le\dots\le t_k$, $c_1,\dots,c_k\in\R^d$, $k\ge1$.
Define
\[
\SMALL Z_n=\sum_{\ell=1}^k c_\ell^T(M_n(t_\ell)-M_n(t_{\ell-1})), \quad
Z=\sum_{\ell=1}^k c_\ell^T(W(t_\ell)-W(t_{\ell-1})).
\]
We must show that $Z_n\to_{\mu_n} Z$.

Now $Z=N(0,V)$ where $V=\sum_{\ell=1}^k c_\ell^T\Sigma c_\ell(t_\ell-t_{\ell-1})$.
Also,
\[
Z_n=n^{-1/2}\sum_{\ell=1}^k c_\ell^T \sum_{j=[nt_{\ell-1}]}^{[nt_\ell]-1}m_n\circ f_n^j=\sum_{j=1}^{[nt_k]}X_{n,j},
\]
where $X_{n,j}=n^{-1/2}d_{n,j}^Tm_n\circ f_n^{[nt_k]-j}$ for
appropriate choices of 
$d_{n,j}\in\{c_1,\dots,c_k\}$.

Define $\cG_{n,j}=f_n^{-([nt_k]-j)}\cM_n$ for $1\le j\le [nt_k]$.
The same calculation as in the proof of Proposition~\ref{prop-mart}
shows that 
$\{X_{n,j},\,\cG_{n,j};\,1\le j\le [nt_k]\}$ is a martingale difference array.
That is, 
$\cG_{n,j}\subset \cG_{n,j+1}$ for all $1\le j\le [nt_k]-1$,
$X_{n,j}$ is $\cG_{n,j}$-measurable for all $1\le j\le [nt_k]$,
and
$\E(X_{n,j+1}|\cG_{n,j})=0$ for all $1\le j\le [nt_k]-1$.

We now apply a CLT for martingale difference arrays~\cite[Theorem~18.1]{Billingsley99}.
(See also~\cite[Theorem~2.3]{McLeish74}.)
To show that $Z_n\to_d N(0,V)$ it suffices to show that
\begin{itemize}
\item[(B1)]
$\sum_{j=1}^{[nt_k]}\E(X_{n,j}^2|\cG_{n,j-1})\to_{\mu_n} V$ as $n\to\infty$.
\item[(B2)] $\lim_{n\to\infty}\sum_{j=1}^{[nt_k]}\E(X_{n,j}^21_{\{|X_{n,j}|\ge\eps\}})=0$ for all $\eps>0$.
\end{itemize}

Arguing as in the proof of Proposition~\ref{prop-mart},
\[
\E(X_{n,j}^2|\cG_{n,j-1})=n^{-1}\E((d_{n,j}^Tm_n)^2\circ f^{[nt_k]-j}|\cG_{n,j-1})=n^{-1}\{U_nL_n(d_{n,j}^Tm_n)^2\}\circ f_n^{[nt_k]-j}.
\]
Hence
\begin{align*}
& \sum_{j=1}^{[nt_k]} \E(X_{n,j}^2|\cG_{n,j-1})
= n^{-1}\sum_{j=1}^{[nt_k]} \{U_nL_n(d_{n,j}^Tm_n)^2\}\circ f_n^{[nt_k]-j}
\\ & \qquad 
 = n^{-1}\sum_{\ell=1}^k \sum_{j=[nt_{\ell-1}]}^{[nt_\ell]-1}\{U_nL_nc_\ell^T (m_nm_n^T)c_\ell\}\circ f_n^j
\\ & \qquad = \sum_{\ell=1}^k c_\ell^T\Big(n^{-1}\sum_{j=0}^{[nt_\ell]-1}\{U_nL_n(m_nm_n^T)\}\circ f_n^j-
n^{-1}\sum_{j=0}^{[nt_{\ell-1}]-1}\{U_nL_n(m_nm_n^T)\}\circ f_n^j\Big)c_\ell
\end{align*}
which converges in probability to $V$.  This proves~(B1).

Next, $|X_{n,j}|\le Kn^{-1/2}|m_n\circ f_n^{[nt_k]-j}|$ where $K=\max\{|c_1|,\dots,|c_k|\}$.
Hence $X_{n,j}^21_{\{|X_{n,j}|\ge\eps\}}\le K^2n^{-1}(|m_n|^21_{\{|m_n|\ge\eps'n^{1/2}\}})\circ f_n^{[nt_k]-j}$ where $\eps'=\eps/K$, and
\begin{align*}
\sum_{j=1}^{[nt_k]}\E(X_{n,j}^21_{\{|X_{n,j}|\ge\eps\}}) & \le K^2n^{-1}
\sum_{j=1}^{[nt_k]}\E(|m_n|^21_{\{|m_n|\ge\eps'n^{1/2}\}})
\\ & =K^2n^{-1}[nt_k]\E(|m_n|^21_{\{|m_n|\ge\eps'n^{1/2}\}}),
\end{align*}
which converges to zero by uniform integrability of $\{|m_n|^2\}$.
This proves (B2) and completes the proof that
$Z_n\to_{\mu_n} Z$, showing that finite-dimensional distributions converge.

\vspace{1ex}
\noindent
{\bf Tightness}
Tightness of $\{M_n\}$ in $D([0,\infty),\R^d)$ is equivalent to
tightness of each coordinate of $\{M_n\}$ in $D([0,T],\R)$ for each $T>0$.
Hence we fix $T>0$, and assume
without loss that $M_n$ is $\mathbb{R}$-valued and that $\frac{1}{n}\sum_{k=0}^{[nt]-1}m_n^2\circ f_n^k\to_{\mu_n}t\sigma^2$ for some $\sigma^2\ge0$.
Since $\Sigma$ is nonzero, we can choose coordinates so that $\sigma^2>0$ in each coordinate.
Without loss $\sigma^2=1$.

Since $M_n(0)\equiv 0$, proving tightness of $\{M_n\}$ is equivalent~\cite[Theorem~7.3]{Billingsley99} to showing that 
\begin{align}\label{eq-tight}
\lim_{\delta\to 0}\limsup_{n\to\infty} 
\mu_n\Bigg(\sup_{\overset{0\leq s,t\leq T}{|s-t|\leq \delta}}|M_n(t)-M_n(s)|>\epsilon\Bigg)=0\quad
\text{for every $\eps>0$}.
\end{align}

Define 
$M_n^-(t)=\sum_{j=1}^{[nt]}\xi_{n,j}$ where
$\xi_{n,j}=n^{-1/2}m_n\circ f_n^{[nT]-j}$.
We claim that the hypotheses of~\cite[Theorem~18.2]{Billingsley99} are satisfied and hence in particular that $\{M_n^-\}$ is tight.
It follows that condition~\eqref{eq-tight} is satisfied with $M_n$ replaced by 
$M_n^-$.
But $M_n(t)-M_n(s)=M_n^-(u_{n,s})-M_n^-(u_{n,t})$ where
$u_{n,t}\in[0,T]$ is such that $[nu_{n,t}]=[nT]-[nt]$.
Hence
\[
\sup_{\overset{0\le s,t\le T}{|s-t|\le\delta}}  |M_n(t)-M_n(s)| \le \sup_{\overset{0\le s,t\le T}{|s-t|\le\delta+\frac2n}}|M_n^-(t)-M_n^-(s)|,
\]
and the result follows.

It remains to verify the claim.   Consider the martingale difference array 
$\{X_{n,j},\cG_{n,j};\,1\le j\le [nT]\}$ where
$X_{n,j}=n^{-1/2}m_n\circ f_n^{[nT]-j}$,
$\cG_{n,j}=f_n^{-([nT]-j)}\cM_n$.
By~\cite[Theorem~18.2]{Billingsley99}, it suffices to show that for each $t\in[0,T]$,
\begin{itemize}
\item[(B3)]
$\sum_{j=1}^{[nt]}\E(X_{n,j}^2|\cG_{n,j-1})\to_{\mu_n} t$ as $n\to\infty$.
\item[(B4)] $\lim_{n\to\infty}\sum_{j=1}^{[nt]}\E(X_{n,j}^21_{\{|X_{n,j}|\ge\eps\}})=0$ for all $\eps>0$.
\end{itemize}
These are proved in exactly the same way as conditions~(B1) and~(B2) above.

\vspace{1ex}
\noindent{\bf The completely degenerate case}
When $\Sigma=0$, we consider the direct product of the underlying dynamics with a simple symmetric random walk on the integers.  The product system
leads to a WIP with one nondegenerate direction and the result reduces to the
case $\Sigma\neq0$.

More precisely, let $\Lambda'=\{\pm1\}^{\N}$ with fair ($p=q=1/2$) Bernoulli measure $\mu'$, and consider the one-sided shift
$f:\Lambda'\to\Lambda'$ and observable $m':\Lambda'\to \{\pm1\}$ where
$m'(x)=x_0$.   The process $M_n'(t)=n^{-1/2}\sum_{j=0}^{[nt]-1}m'\circ f'^j$
converges in $D([0,\infty),\R)$ to Brownian motion $W'$ with variance $1$.
Define the family of product systems $\Lambda_n''=\Lambda_n\times\Lambda'$,
$\mu_n''=\mu_n\times\mu'$, 
$f_n''=f_n\times f':\Lambda_n''\to\Lambda_n''$, $m_n''=m_n\oplus m':\Lambda_n''\to\R^{d+1}$,
$M_n''=M_n\oplus M_n'\in D([0,\infty),\R^{d+1})$.
Let $U_n''$ and $L_n''$ denote the Koopman and transfer operators corresponding
to $f_n''$.
An easy calculation shows that
$U''L''(m''m''^T)=UL(mm^T)\oplus 1$ and hence
$\{m_n''\}$ satisfies the hypotheses of the theorem
with $\Sigma''=0\oplus 1$.  Consequently, $M_n''\to_{\mu_n''} W''=0\oplus W'$
in $D([0,\infty),\R^{d+1})$.
In particular, $M_n\to_{\mu_n}0$.
\end{proof}

\paragraph{Acknowledgements}
This research was supported in part by a European Advanced Grant {\em StochExtHomog} (ERC AdG 320977).

\def\polhk#1{\setbox0=\hbox{#1}{\ooalign{\hidewidth
  \lower1.5ex\hbox{`}\hidewidth\crcr\unhbox0}}}

\end{document}